\documentclass{amsart}

\usepackage{amssymb,url,amsthm}
\usepackage{mathtools}
\usepackage{hyperref}
\usepackage{color,graphicx}
\usepackage{subcaption}

\newtheorem{theorem}{Theorem}[section]
\newtheorem{lemma}[theorem]{Lemma}

\newtheorem{algorithm}[theorem]{Algorithm}
\newtheorem{remark}[theorem]{Remark}
\newtheorem{proposition}[theorem]{Proposition}
\newtheorem{definition}[theorem]{Definition}
\newtheorem{corollary}[theorem]{Corollary}

\newcommand{\rd}{\, \mathrm{d}}
\newcommand{\re}{\mathrm{e}}
\newcommand{\ri}{\mathrm{i}}
\newcommand{\bsgamma}{\boldsymbol{\gamma}}
\newcommand{\bsdelta}{\boldsymbol{\Delta}}
\newcommand{\bszero}{\boldsymbol{0}}

\newcommand{\bsk}{\boldsymbol{k}}
\newcommand{\bsell}{\boldsymbol{\ell}}
\newcommand{\bsx}{\boldsymbol{x}}
\newcommand{\bsy}{\boldsymbol{y}}
\newcommand{\bsz}{\boldsymbol{z}}
\newcommand{\ran}{\mathrm{ran}}
\newcommand{\rms}{\mathrm{rms}}
\newcommand{\wor}{\mathrm{wor}}

\newcommand{\EE}{\mathbb{E}}
\newcommand{\NN}{\mathbb{N}}
\newcommand{\PP}{\mathbb{P}}
\newcommand{\RR}{\mathbb{R}}
\newcommand{\ZZ}{\mathbb{Z}}

\newcommand{\Pcal}{\mathcal{P}}
\newcommand{\Zcal}{\mathcal{Z}}
 
\newcommand{\Mod}[1]{\ (\mathrm{mod}\ #1)}

\allowdisplaybreaks

\title[A randomized lattice rule without CBC construction]{A randomized lattice rule without component-by-component construction}
\author{Takashi Goda}
\address{School of Engineering, The University of Tokyo, 7-3-1 Hongo, Bunkyo-ku, Tokyo 113-8656, Japan.}
\email{goda@frcer.t.u-tokyo.ac.jp}
\thanks{This work is supported by JSPS KAKENHI Grant Number 23K03210.}

\date{\today}

\begin{document}

\begin{abstract}
    We study the multivariate integration problem for periodic functions from the weighted Korobov space in the randomized setting. We introduce a new randomized rank-1 lattice rule with a randomly chosen number of points, which avoids the need for component-by-component construction in the search for good generating vectors while still achieving nearly the optimal rate of the randomized error. Our idea is to exploit the fact that at least half of the possible generating vectors yield nearly the optimal rate of the worst-case error in the deterministic setting. By randomly choosing generating vectors $r$ times and comparing their corresponding worst-case errors, one can find one generating vector with a desired worst-case error bound with a very high probability, and the (small) failure probability can be controlled by increasing $r$ logarithmically as a function of the number of points. Numerical experiments are conducted to support our theoretical findings.
\end{abstract}

\maketitle

\section{Introduction}
We study the problem of numerical integration for functions defined over the multi-dimensional unit cube. For a dimension $d$ and a function $f: [0,1]^d\to \RR$, we denote the integral of $f$ by
\[ I(f):=\int_{[0,1]^d}f(\bsx)\rd \bsx. \]
Throughout this paper, we assume that $f$ is one-periodic in each variable and belongs to the \emph{weighted Korobov space}, denoted as $H_{\alpha,\bsgamma}$, whose precise definition will be given in Section~\ref{subsec:korobov}. For now, it suffices to understand that the space $H_{\alpha,\bsgamma}$ is a normed space parametrized by the smoothness parameter $\alpha\geq 0$ and the weight parameters $\bsgamma=(\gamma_1,\gamma_2,\ldots)\in \RR_{\geq 0}^{\NN}$, where we denote the norm by $\|\cdot\|_{\alpha,\bsgamma}$.

We aim to approximate $I(f)$ using a quasi-Monte Carlo (QMC) rule over an $N$-element point set $P_N\subset [0,1]^d$:
\begin{align}\label{eq:qmc_rule}
    Q_{P_N}(f):=\frac{1}{N}\sum_{\bsx\in P_{N}}f(\bsx).
\end{align}
In the deterministic setting, as the most significant error criterion, the worst-case error of the QMC rule over $P_N$ in the space $H_{\alpha,\bsgamma}$ is defined as
\[ e^{\wor}(H_{\alpha,\bsgamma}; Q_{P_N}) := \sup_{\substack{f\in H_{\alpha,\bsgamma}\\ \|f\|_{\alpha,\bsgamma}\leq 1}}\left| I(f)-Q_{P_N}(f)\right|. \]
It is well-known that QMC rules using rank-1 lattice point sets, referred to as \emph{rank-1 lattice rules} for short, achieve nearly the optimal rate of the worst-case error of $O(N^{-\alpha})$, see, for instance, \cite[Chapters~2 \& 3]{dick2022lattice}. Here a rank-1 lattice point set, denoted as $P_{N,\bsz}$, is determined by the number of points $N$ and the generating vector $\bsz\in \{1,\ldots,N-1\}^d$. The precise definition of rank-1 lattice point sets is deferred to Section~\ref{subsec:lattice}. To search for good generating vectors, component-by-component (CBC) construction has been often used \cite{dick2004convergence,dick2006good,kuo2003component,sloan2002component}, which is a greedy algorithm that recursively looks for all possible candidates from the set $\{1,\ldots,N-1\}$ for $z_{j+1}$ while keeping the earlier components $z_1,\ldots,z_j$ fixed. As shown in  \cite{nuyens2006fast}, with the fast Fourier transform, the total construction cost of the CBC construction is given by $O(dN\log N)$. Stability of rank-1 lattice rules obtained by the CBC construction (with given $\alpha$ and $\bsgamma$) in the weighted Korobov spaces with the different smoothness and weight parameters has been studied in \cite{dick2021stability}.

In this paper, we are concerned with the randomized setting, where instead of fixing a point set $P_N$ deterministically, we allow for randomness in both the number of points $N$ and the point sets $P_N$. Specifically, we consider a randomized cubature rule defined by a pair of a probability space $(\Omega,\Sigma,\mu)$ and a family $(Q^{\omega})_{\omega\in \Omega}$ such that, for each fixed $\omega\in \Omega$, the corresponding element $Q^{\omega}$ is a deterministic cubature rule of the form \eqref{eq:qmc_rule}. In this context, we adopt two error criteria: the worst-case randomized error and the worst-case root-mean-squared error (RMSE), defined as
\[ e^{\ran}(H_{\alpha,\bsgamma}; (Q^{\omega})_{\omega\in \Omega}) := \sup_{\substack{f\in H_{\alpha,\bsgamma}\\ \|f\|_{\alpha,\bsgamma}\leq 1}}\int_{\Omega}\left| I(f)-Q^{\omega}(f)\right|\rd \mu(\omega) , \]
and
\[ e^{\rms}(H_{\alpha,\bsgamma}; (Q^{\omega})_{\omega\in \Omega}) := \sup_{\substack{f\in H_{\alpha,\bsgamma}\\ \|f\|_{\alpha,\bsgamma}\leq 1}}\left( \int_{\Omega}\left| I(f)-Q^{\omega}(f)\right|^2\rd \mu(\omega) \right)^{1/2}, \]
respectively.
One major advantage of a randomized rule over a deterministic rule is that an error estimator to assess the accuracy is often available for individual integrals $I(f)$. Furthermore, the optimal convergence rates of the randomized error criteria may improve compared to the deterministic worst-case error \cite{novak1988deterministic,novak2008tractability,novak2010tractability}. 

Building upon the classical work by Bahvalov \cite{bahvalov1961estimate}, Kritzer et al.\  proved in \cite{kritzer2019lattice} that, by choosing the generating vector $\bsz$ randomly from a set of good candidates, denoted by $\Zcal_{N,\eta,\alpha,\bsgamma}\subseteq \{1,\ldots,N-1\}^d$ for an additional parameter $\eta\in (0,1)$, with a randomly chosen prime $N\in \{\lceil M/2\rceil+1,\ldots,M\}$, a randomized rank-1 lattice rule achieves nearly the optimal rate of the randomized error of $O(M^{-\alpha-1/2})$ in the space $H_{\alpha,\bsgamma}$ for $\alpha> 1/2$. Here, the implicit constant is bounded independently of the dimension $d$ if $\sum_{j=1}^{\infty}\gamma_j^{1/\alpha}<\infty$. Moreover, with an additional random shift applied to a randomized rank-1 lattice point set, this result is extended to the RMSE in $H_{\alpha,\bsgamma}$ not only for $\alpha> 1/2$ but also for $0<\alpha\leq 1/2$. However, it is challenging to take a random sample from the uniform distribution over $\Zcal_{N,\eta,\alpha,\bsgamma}$ whose size is exponentially large in $d$, i.e., $|\Zcal_{N,\eta,\alpha,\bsgamma}|\geq \eta (N-1)^d$, making implementation of the algorithm difficult in practice.

There have been some recent works addressing this issue. In \cite{dick2022component}, Dick et al.\ introduced a randomized CBC construction to randomly choose a good generating vector $\bsz$ for the case $\alpha>1/2$. This is an online algorithm, where each sample $\bsz$ needs to be constructed at runtime with a computational cost of $O(d N \log N)$ by the fast Fourier transform approach from \cite{nuyens2006fast} for a randomly chosen $N$. In \cite{kuo2023random}, Kuo et al.\ eliminated the randomness in choosing a generating vector and demonstrated that, with a carefully constructed single generating vector $\bsz$, a randomized lattice rule with a randomly chosen prime $N\in \{\lceil M/2\rceil+1,\ldots,M\}$, which is the only source of randomness in the algorithm, still achieves nearly the optimal rate of the randomized error. The corresponding generating vector can be constructed offline ahead of time, while the CBC algorithm requires a computational cost of $O(dM^4/\log M)$. This result is further extended to the RMSE for all $\alpha>0$ by applying a random shift in \cite{nuyens2023randomised}. In passing, a randomized trapezoidal rule with random $N$ to approximate integrals with respect to one-dimensional Gaussian measure has been investigated in \cite{goda2023randomizing}.

In this paper, we revisit the property that $|\Zcal_{N,\eta,\alpha,\bsgamma}|\geq \eta (N-1)^d$, indicating that, with $\eta=1/2$, at least half of the possible choices for the generating vector $\bsz$ from $\{1,\ldots,N-1\}^d$ are already good. In fact, the set $\Zcal_{N,\eta,\alpha,\bsgamma}$ is defined as a collection of the generating vectors whose corresponding lattice rules achieve nearly the optimal order of the deterministic worst-case error in $H_{\alpha,\bsgamma}$. This property was previously shown in \cite[Theorem~2]{dick2006good} and has been utilized in a recent work by the author and L'Ecuyer \cite{goda2022construction}, where the median of several independently and randomly chosen rank-1 lattice rules was shown to be able to adjust to the smoothness and weight parameters of $H_{\alpha,\bsgamma}$ without any need to specify them. In this paper, we exploit the same property in the following way to introduce yet another randomized lattice rule: by independently and randomly drawing generating vectors $r$ times from the set $\{1,\ldots,N-1\}^d$ and then comparing the worst-case errors of rank-1 lattice rules with these generating vectors in $H_{\alpha,\bsgamma}$, the generating vector that minimizes the worst-case error among them belongs to the set $\Zcal_{N,\eta,\alpha,\bsgamma}$ with a very high probability $1-(1-\eta)^r$. Although there is an exponentially small failure probability, its contribution to the overall error bound can be mitigated through averaging over randomness with an appropriate choice of $r$.

The rest of this paper is organized as follows. In Section~\ref{sec:pre}, we provide the definitions of the weighted Korobov space $H_{\alpha,\bsgamma}$ and rank-1 lattice rules. Section~\ref{sec:new} introduces a new online randomized rank-1 lattice rule and proves that it achieves nearly the optimal rate of the randomized error of $O(M^{-\alpha-1/2})$ in $H_{\alpha,\bsgamma}$ for $\alpha> 1/2$. Furthermore, we extend this result to the RMSE for $\alpha> 1/2$ by applying a random shift. Finally, we conclude this paper with numerical experiments in Section~\ref{sec:numerics}, which provide empirical support for our theoretical findings.
\vspace{1em}

\noindent \textbf{Notation.} Throughout this paper, we denote the set of integers by $\ZZ$ and the set of positive integers by $\NN$. We use the shorthand $\{1{:}d\}$ to denote the set $\{1,\ldots,d\}$. For any $u\subseteq \{1{:}d\}$ and a vector $\bsx\in \RR^d$, we use the notation $-u=\{1{:}d\}\setminus u$ to denote the complement of $u$, and $\bsx_u=(x_j)_{j\in u}$ to represent the sub-vector of $\bsx$ consisting of its components indexed by $u$. The vector with all elements being $0$ is denoted as $\bszero$. When considering the probability of a stochastic event $A$ with respect to a random variable $y$, we denote it by $\PP_y[A]$. If the random variable is clear from the context, we omit the subscript and simply write $\PP[A]$. This convention is also applied to the expectation $\EE$. Furthermore, we denote the indicator function of $A$ by $\chi(A)$, i.e., $\chi(A)=1$ if $A$ holds and $\chi(A)=0$ otherwise.

\section{Preliminaries}\label{sec:pre}

\subsection{Weighted Korobov space}\label{subsec:korobov}
Let us consider a function $f: [0,1]^d\to \RR$ that is one-periodic in each variable, and assume that it has an absolutely convergent Fourier series:
\begin{align}\label{eq:fourier_series}
    f(\bsx)=\sum_{\bsk\in \ZZ^d}\hat{f}(\bsk)\, \re^{2\pi i \bsk\cdot \bsx},
\end{align}
where $\cdot$ represents the inner product between two vectors in $\RR^d$, and 
\[ \hat{f}(\bsk) := \int_{[0,1]^d}f(\bsx)\, \re^{-2\pi \ri \bsk\cdot \bsx}\rd \bsx\]
denotes the $\bsk$-th Fourier coefficient of $f$.

For a real number $\alpha>1/2$ and a sequence of non-negative real numbers $\bsgamma=(\gamma_1,\gamma_2,\ldots)\in \RR_{\geq 0}^{\NN}$, let $r_{\alpha,\bsgamma}: \ZZ^d\to \RR_{\geq 0}$ be defined by
\[ r_{\alpha,\bsgamma}(\bsk):=\prod_{\substack{j=1\\ k_j\neq 0}}^{d}\frac{|k_j|^{\alpha}}{\gamma_j},\]
where we set the empty product to $1$ and $|k_j|^{\alpha}/\gamma_j=\infty$ for any $k_j\in \ZZ$ if $\gamma_j=0$. 
Later in this paper, more precisely in our theoretical analysis on randomized error criteria in Section~\ref{subsec:error_bounds}, we assume that each $\gamma_j$ is bounded above by $1$. This ensures that $r_{\alpha,\bsgamma}(\bsk)\geq 1$ for all $\bsk$.

Then the weighted Korobov space $H_{\alpha,\bsgamma}$ is defined as a reproducing kernel Hilbert space with the inner product
\[ \langle f,g\rangle_{\alpha,\bsgamma}:= \sum_{\bsk\in \ZZ^d}\hat{f}(\bsk)\overline{g(\bsk)}(r_{\alpha,\bsgamma}(\bsk))^2, \]
and the reproducing kernel
\[ K_{\alpha,\bsgamma}(\bsx,\bsy)=\sum_{\bsk\in \ZZ^d}\frac{\re^{2\pi \ri \bsk\cdot (\bsx-\bsy)}}{(r_{\alpha,\bsgamma}(\bsk))^2}.\]
The induced norm is simply denoted by $\|f\|_{\alpha,\bsgamma}=\sqrt{\langle f,f\rangle_{\alpha,\bsgamma}}.$
Here we assume that $\hat{f}(\bsk)=0$ if $r_{\alpha,\bsgamma}(\bsk)=\infty$ and we interpret $\infty\cdot 0$ as equal to $0$ so that the corresponding frequencies do not contribute to the inner product or the norm of the function space.

Note that even for $0\leq \alpha\leq 1/2$, the weighted Korobov space can be defined as above, and it holds that $H_{\alpha,\bsgamma}\subseteq L_2([0,1]^d)$ with equality for $\alpha=0$ and $\gamma_j>0$ for all $j\in \NN$. For $0 \leq \alpha \leq 1/2$, the functions in $H_{\alpha,\bsgamma}$ are not necessarily periodic, unlike the case for $\alpha > 1/2$ where we assume periodicity. Moreover, the Fourier series of $f\in H_{\alpha,\bsgamma}$ does not necessarily converge absolutely if $0 \leq \alpha \leq 1/2$, and the equality in \eqref{eq:fourier_series} should be understood to hold almost everywhere. As we focus on the case $\alpha>1/2$ throughout this paper, the Fourier series of $f\in H_{\alpha,\bsgamma}$ always converges absolutely.

It is evident from the definition of $H_{\alpha,\bsgamma}$ that the parameter $\alpha$ controls how fast the Fourier coefficients decay, and that the sequence $\bsgamma$ moderates the relative importance of different variables, as considered in \cite{sloan1998when}. Moreover, the parameter $\alpha$ is directly related to the smoothness of periodic functions; for $\alpha\in \NN$, it holds that
\[ \|f\|_{\alpha,\bsgamma}^2 = \sum_{u\subseteq \{1{:}d\}}\frac{1}{(2\pi)^{2\alpha |u|}\gamma_u^2}\int_{[0,1]^{|u|}}\left| \int_{[0,1]^{d-|u|}}\frac{\partial^{\alpha |u|}}{\prod_{j\in u}\partial x_j^{\alpha}}f(\bsx)\rd \bsx_{-u}\right|^2\rd \bsx_u, \]
where we define $\gamma_u=\prod_{j\in u}\gamma_j$ for any $u\subseteq \{1{:}d\}$, with the empty product set to $1$. We refer to \cite[Propositions~2.4 \& 2.17]{dick2022lattice} for this equality.

\subsection{Rank-1 lattice rule}\label{subsec:lattice}
For comprehensive information on lattice rules, we refer to the books by Niederreiter \cite{niederreiter1992random}, Sloan and Joe \cite{sloan1994lattice}, and Dick et al.\ \cite{dick2022lattice}. First, a rank-1 lattice point set is defined as follows.
\begin{definition}[Rank-1 lattice point set]
    For $N\in \NN$ with $N\geq 2$, let $\bsz=(z_1,\ldots,z_d)\in \{1,\ldots,N-1\}^d$ be a vector. The rank-1 lattice point set $P_{N,\bsz}$ is defined by
    \[ P_{N,\bsz}:=\left\{ \left(\left\{ \frac{nz_1}{N}\right\},\ldots,\left\{ \frac{nz_d}{N}\right\}\right) \,\colon \, 0\leq n<N\right\},\]
    where $\{x\}$ denotes the fractional part of $x$, i.e., $\{x\}:=x-\lfloor x\rfloor$ for real $x$. The QMC rule using $P_{N,\bsz}$ as a point set is called the rank-1 lattice rule with generating vector $\bsz$.
\end{definition}

The dual lattice of a rank-1 lattice point set, defined below, plays a key role in analyzing the integration error of a rank-1 lattice rule.
\begin{definition}[Dual lattice]
    For $N\in \NN$ with $N\geq 2$ and $\bsz\in \{1,\ldots,N-1\}^d$, the dual lattice of the rank-1 lattice point set $P_{N,\bsz}$ is defined as
    \[ P_{N,\bsz}^{\perp}:=\left\{ \bsk\in \ZZ^d \,\colon \, \bsk\cdot \bsz\equiv 0 \pmod N\right\}.\]
\end{definition}

The \emph{character property}, as described in the following lemma, is an essential characteristic of a rank-1 lattice point set. It indicates that a rank-1 lattice rule integrates the $\bsk$-th Fourier mode exactly if and only if either $\bsk=\bszero$ or $\bsk\not\in P_{N,\bsz}^{\perp}$. The proof can be found, for instance, in \cite[Lemma~1.9]{dick2022component}.
\begin{lemma}\label{lem:character}
    Let $N\in \NN$ with $N\geq 2$ and $\bsz\in \{1,\ldots,N-1\}^d$. For any $\bsk\in \ZZ^d$, it holds that
    \[ \frac{1}{N}\sum_{\bsx\in P_{N,\bsz}}\re^{2\pi \ri \bsk\cdot \bsx}=\begin{cases} 1 & \text{if $\bsk\in P_{N,\bsz}^{\perp}$,} \\ 0 & \text{otherwise.}\end{cases} \]
\end{lemma}

It is known that the worst-case error of a rank-1 lattice rule in the weighted Korobov space can be explicitly given, as stated in the following lemma. The proof is available, for instance, in \cite[Section~2.5]{dick2022component}.
\begin{lemma}\label{lem:worst-case}
    Let $\alpha>1/2$ be a real number and $\bsgamma=(\gamma_1,\gamma_2,\ldots)\in \RR_{\geq 0}^{\NN}$ be a sequence of non-negative weights. For a rank-1 lattice rule with $N\in \NN$, $N\geq 2$, and $\bsz\in \{1,\ldots,N-1\}^d$, it holds that
    \[ \left(e^{\wor}(H_{\alpha,\bsgamma}; Q_{P_{N,\bsz}})\right)^2=\sum_{\bsk\in P_{N,\bsz}^{\perp}\setminus \{\bszero\}}\frac{1}{(r_{\alpha,\bsgamma}(\bsk))^2}.\]
    Moreover, if $\alpha\in \NN$, this formula can be written in the form
    \[ \left(e^{\wor}(H_{\alpha,\bsgamma}; Q_{P_{N,\bsz}})\right)^2=-1+\frac{1}{N}\sum_{\bsx\in P_{N,\bsz}}\prod_{j=1}^{d}\left[ 1+\gamma_j\frac{(-1)^{\alpha+1}(2\pi)^{2\alpha}}{(2\alpha)!}B_{2\alpha}(x_j)\right], \]
    which is computable in $O(dN)$ operations, where $B_{2\alpha}$ denotes the Bernoulli polynomial of degree $2\alpha$.
\end{lemma}

As mentioned in the previous section, the set $\Zcal_{N,\eta,\alpha,\bsgamma}$, with a sufficient cardinality, is a collection of the generating vectors $\bsz$ for which the corresponding rank-1 lattice rules achieve nearly the optimal order of $e^{\wor}(H_{\alpha,\bsgamma}; Q_{P_{N,\bsz}})$. The following proposition restates the result given in \cite[Theorem~2]{dick2006good}.
\begin{proposition}\label{prop:prob_1}
    Let $\alpha>1/2$ be a real number and $\bsgamma=(\gamma_1,\gamma_2,\ldots)\in \RR_{\geq 0}^{\NN}$ be a sequence of non-negative weights. Let $N\geq 2$ be a prime and $\eta\in (0,1)$. For $1/2\leq \lambda<\alpha$, we define
    \[ B_{N,\eta,\alpha,\bsgamma,\lambda}:=\left(\frac{1}{(1-\eta) (N-1)}\prod_{j=1}^{d}\left(1+2\gamma_j^{1/\lambda}\zeta(\alpha/\lambda)\right)\right)^{\lambda}, \]
    where $\zeta$ denotes the Riemann-zeta function, and further define
    \begin{align*}
        \Zcal_{N,\eta,\alpha,\bsgamma} & :=\left\{\bsz\in \{1,\dots,N-1\}^d \,\colon \,\right.\\
        & \qquad \qquad \left.e^{\wor}(H_{\alpha,\bsgamma};P_{N,\bsz})\leq B_{N,\eta,\alpha,\bsgamma,\lambda} \quad \text{for all $1/2\leq \lambda<\alpha$}\right\}.
    \end{align*}
    Suppose that $\bsz$ is a random variable following the uniform distribution over the set $\{1,\ldots,N-1\}^d$, Then it holds that
    \[ \PP\left[ \bsz\in \Zcal_{N,\eta,\alpha,\bsgamma}\right] \geq \eta, \]
    for any $0<\eta<1$. Equivalently, we have $|\Zcal_{N,\eta,\alpha,\bsgamma}|\geq \eta (N-1)^d$.
\end{proposition}

In this proposition, the bound $B_{N,\eta,\alpha,\bsgamma,\lambda}$ on the worst-case error approaches the rate $O(N^{-\alpha})$ arbitrarily closely as a function of $N$ when $\lambda\to \alpha^{-}$. Moreover, it can be shown that this bound is further bounded independently of the dimension $d$ for all $1/2\leq \lambda<\alpha$ if $\sum_{j=1}^{\infty}\gamma_j^{1/\alpha}<\infty$, since then
\[ \prod_{j=1}^{d}\left(1+2\gamma_j^{1/\lambda}\zeta(\alpha/\lambda)\right)\leq \re^{2\zeta(\alpha/\lambda)\sum_{j=1}^{d}\gamma_j^{1/\lambda}}\leq \re^{2\zeta(\alpha/\lambda)\sum_{j=1}^{\infty}\gamma_j^{1/\alpha}},\]
but, unfortunately, $\zeta(\alpha/\lambda)\to \infty$ if $\lambda\to \alpha^{-}$.

\section{A randomized lattice rule}\label{sec:new}

With the above preliminaries in mind, we introduce a new randomized lattice rule in Section~\ref{subsec:alg} and then analyze two error criteria, namely, the randomized error and the RMSE, in the weighted Korobov space $H_{\alpha,\bsgamma}$ in Section~\ref{subsec:error_bounds}. Additional remarks are collected in Section~\ref{subsec:remarks}.

\subsection{Algorithm and basic properties}\label{subsec:alg}
Let us consider the following algorithm:
\begin{algorithm}\label{alg:without}
Let $\alpha>1/2$ be a real number and $\bsgamma=(\gamma_1,\gamma_2,\ldots)\in \RR_{\geq 0}^{\NN}$ be a sequence of non-negative weights. Given $r,M\in \NN$, with $M\geq 2$, proceed as follows:
\begin{enumerate}
    \item Randomly draw $N\in \NN$ from the uniform distribution over the set 
    \[ \Pcal_M:=\left\{ p\in \NN_p \,\colon \, \lceil M/2\rceil <p\leq M \right\}, \]
    where $\NN_p$ denotes the set of prime numbers.
    \item Independently and randomly draw $r$ vectors, denoted by $\bsz_1,\ldots,\bsz_r$, from the uniform distribution over the set $\{1,\ldots,N-1\}^d.$
    \item Choose $\bsz^*$ that minimizes the worst-case error among $\bsz_1,\ldots,\bsz_r$:
    \[ \bsz^*:=\arg\min_{1\leq j\leq r}e^{\wor}(H_{\alpha,\bsgamma}; P_{N,\bsz_j}). \]
\end{enumerate}
\end{algorithm}

Our new, online randomized rank-1 lattice rule (without shift), denoted simply by $A^{\ran}_{r,M}$, is defined as a family of $Q_{P_{N,\bsz^*}}$, where $N$ and $\bsz^*$ are random variables. Here, $N$ is always a prime number, as it is drawn from the set $\Pcal_M$, consisting of primes in the range $\lceil M/2\rceil < p \leq M$. ``Online'' refers to the fact that the selection of the generating vectors needs to be performed at runtime, rather than being predetermined beforehand. Specifically, the realizations of $N$ and $\bsz^*$ are generated as described in Algorithm~\ref{alg:without} with inputs $\alpha,\bsgamma,r,M$.
With an additional random shift, in which case our new randomized lattice rule is denoted by $\tilde{A}^{\ran}_{r,M}$, we draw $\bsdelta=(\Delta_1,\ldots,\Delta_d)$ randomly from the uniform distribution over $[0,1]^d$ and $Q_{P_{N,\bsz^*}}$ is replaced by $Q_{P_{N,\bsz^*}+\bsdelta}$, where
\[ P_{N,\bsz^*}+\bsdelta := \left\{ \left(\{x_1+\Delta_1\}, \ldots, \{x_d+\Delta_d\}\right)\,\colon \, \bsx\in P_{N,\bsz^*} \right\}.\]
Here, we recall that $\{x\}$ denotes the fractional part of a real number $x$. 

Before moving on to the analysis of error criteria, we show some basic properties of our new randomized lattice rule.

\begin{lemma}\label{lem:prob_r}
    Let $\alpha>1/2$ be a real number and $\bsgamma\in \RR_{\geq 0}^{\NN}$ be a sequence of non-negative weights. Furthermore, let $r,M\in \NN$, with $M\geq 2$, be given. For $N$ and $\bsz^*$ being drawn by Algorithm~\ref{alg:without}, it holds that
    \[ \PP_{N,\bsz_1,\ldots,\bsz_r}\left[ \bsz^*\in \Zcal_{N,\eta,\alpha,\bsgamma}\right] \geq 1-(1-\eta)^r, \]
    for any $0<\eta<1$, where $\Zcal_{N,\eta,\alpha,\bsgamma}$ is defined in Proposition~\ref{prop:prob_1}.
\end{lemma}
\begin{proof}
    From the definitions of $\bsz^*$ and $\Zcal_{N,\eta,\alpha,\bsgamma}$, we observe that $\bsz^*\not\in \Zcal_{N,\eta,\alpha,\bsgamma}$ if and only if $\bsz_j\not\in \Zcal_{N,\eta,\alpha,\bsgamma}$ holds for all $1\leq j\leq r$. By using the independence between $\bsz_1,\ldots,\bsz_r$ (conditional on $N$) and Proposition~\ref{prop:prob_1}, we have
    \begin{align*}
        \PP\left[ \bsz^*\in \Zcal_{N,\eta,\alpha,\bsgamma}\right] & = 1- \PP\left[ \bsz^*\not \in \Zcal_{N,\eta,\alpha,\bsgamma}\right] \\
        & = 1- \PP\left[ \bsz_1\not \in \Zcal_{N,\eta,\alpha,\bsgamma},\ldots,\bsz_r\not \in \Zcal_{N,\eta,\alpha,\bsgamma}\right] \\
        & = 1- \EE_N\left[\left(\PP_{\bsz_1\mid N}\left[ \bsz_1\not \in \Zcal_{N,\eta,\alpha,\bsgamma}\right]\right)^r\right]\\
        & = 1- \EE_N\left[\left(1-\PP_{\bsz_1\mid N}\left[ \bsz_1\in \Zcal_{N,\eta,\alpha,\bsgamma}\right]\right)^r\right] \\
        & \geq 1-(1-\eta)^r,
    \end{align*}
    where $\PP_{\bsz_1\mid N}[\bullet]$ denotes the conditional probability with respect to $\bsz_1$ given $N$.
    Hence, the proof is complete.
\end{proof}

\begin{lemma}\label{lem:prob_in_dual}
    For given $r,M\in \NN$, with $M\geq 2$, let $N$ and $\bsz^*$ be drawn according to Algorithm~\ref{alg:without} with any $\alpha>1/2$ and sequence $\bsgamma\in \RR_{\geq 0}^{\NN}$. Then, there exists a constant $c>0$ such that
    \[ \PP_{N,\bsz_1,\ldots,\bsz_r}\left[ \bsk\in P_{N,\bsz^*}^{\perp}\right] \leq c\frac{r}{M}\log(1+\|\bsk\|_{\infty}) \]
    holds for any $\bsk\in \ZZ^d\setminus \{\bszero\}$, where $\|\bsk\|_{\infty}:=\max_{j}|k_j|$ denotes the maximum norm of a vector.
\end{lemma}
\begin{proof}
    From the definition of $\bsz^*$, it is evident that if $\bsk\not\in P_{N,\bsz_j}^{\perp}$ for all $1\leq j\leq r$, then $\bsk\not\in P_{N,\bsz^*}^{\perp}$. By using this trivial fact and the independence between $\bsz_1,\ldots,\bsz_r$ (conditional on $N$), we obtain
    \begin{align*}
    \PP\left[\bsk\in P_{N,\bsz^*}^{\perp}\right] & = 1-\PP\left[\bsk\not \in P_{N,\bsz^*}^{\perp}\right] \\
    & \leq  1-\PP\left[\bsk\not \in P_{N,\bsz_1}^{\perp}, \ldots, \bsk\not \in P_{N,\bsz_r}^{\perp}\right] \\
    & = 1-\EE_{N}\left[ \left(\PP_{\bsz_1\mid N}\left[\bsk\not \in P_{N,\bsz_1}^{\perp}\right] \right)^r \right] \\
    & = 1-\EE_{N}\left[ \left(1-\PP_{\bsz_1\mid N}\left[\bsk \in P_{N,\bsz_1}^{\perp}\right] \right)^r \right] .
    \end{align*}
    Recall that the condition $\bsk \in P_{N,\bsz_1}^{\perp}$ is equivalent to $\bsk\cdot \bsz_1\equiv 0\pmod N$, which holds for all $\bsz_1\in \{1,\ldots,N-1\}^d$ if $\bsk\equiv \bszero \pmod N$, i.e., if $k_j$ is a multiple of $N$ for all $j$. Otherwise, if $\bsk\not\equiv \bszero \pmod N$, i.e., if there exists at least one index $\ell$ such that $k_{\ell}$ is not any multiple of $N$, then the condition $\bsk \in P_{N,\bsz_1}^{\perp}$ is further equivalent to $k_{\ell}z_{1,\ell}\equiv -\bsk_{-\{\ell\}}\cdot \bsz_{1,-\{\ell\}}\pmod N$, which has at most one solution $\tilde{z}_{1,\ell}\in \{1,\dots,N-1\}$ for any $\bsz_{1,-\{\ell\}}\in \{1,\ldots,N-1\}^{d-1}$ since $N$ is a prime. This implies that we have
    \[ \PP_{\bsz_1\mid N}\left[\bsk \in P_{N,\bsz_1}^{\perp}\right] \leq \begin{cases} 1 & \text{if $\bsk\equiv \bszero \Mod{N}$,}\\ (N-1)^{-1} & \text{otherwise.} \end{cases} \]

    Applying this bound, we obtain
    \begin{align*}
    \PP\left[\bsk\in P_{N,\bsz^*}^{\perp}\right] & \leq 1-\frac{1}{|\Pcal_M|}\sum_{\substack{N\in \Pcal_M\\ \bsk\not\equiv \bszero \Mod{N}}}\left(1-\frac{1}{N-1} \right)^r \\
    & = \frac{1}{|\Pcal_M|}\left[\sum_{\substack{N\in \Pcal_M\\ \bsk\equiv \bszero \Mod{N}}}1+\sum_{\substack{N\in \Pcal_M\\ \bsk\not\equiv \bszero \Mod{N}}}\left(1-\left(1-\frac{1}{N-1} \right)^r\right)\right] \\
    & \leq \frac{1}{|\Pcal_M|}\sum_{\substack{N\in \Pcal_M\\ \bsk\equiv \bszero \Mod{N}}}1+\frac{1}{|\Pcal_M|}\sum_{\substack{N\in \Pcal_M\\ \bsk\not\equiv \bszero \Mod{N}}}\frac{r}{N-1}\\
    & \leq \frac{1}{|\Pcal_M|}\sum_{\substack{N\in \Pcal_M\\ \bsk\equiv \bszero \Mod{N}}}1+\frac{r}{\lceil M/2\rceil}.
    \end{align*}
    It has been already known from \cite{dick2022component,kritzer2019lattice,kuo2023random} that the first term on the right-most side above is bounded above by $c'\log(\|\bsk\|_{\infty})/M$ for a constant $c'>0$ independent of $\bsk$. Thus, we have
    \[ \PP\left[\bsk\in P_{N,\bsz^*}^{\perp}\right]\leq c'\frac{\log(\|\bsk\|_{\infty})}{M}+\frac{2r}{M}\leq  c\frac{r}{M}\log(1+\|\bsk\|_{\infty}),\]
    for a constant $c>0$ independent of $\bsk$, which completes the proof.
\end{proof}

\subsection{Error bounds}\label{subsec:error_bounds}
Here we present an upper bound on the worst-case randomized error of our randomized rank-1 lattice rule $A^{\ran}_{r,M}$ in the weighted Korobov space $H_{\alpha,\bsgamma}$, and then extend it to the worst-case RMSE when an additional random shift is applied. Recalling the notation introduced thus far, the worst-case randomized error of $A^{\ran}_{r,M}$ (our randomized lattice rule \emph{without} shift) is expressed as
\[ e^{\ran}(H_{\alpha,\bsgamma}; A^{\ran}_{r,M}) = \sup_{\substack{f\in H_{\alpha,\bsgamma}\\ \|f\|_{\alpha,\bsgamma}\leq 1}}\EE_{N,\bsz_1,\ldots,\bsz_r}\left[ |I(f)-Q_{P_{N,\bsz^*}}(f)|\right], \]
and the worst-case RMSE of $\tilde{A}^{\ran}_{r,M}$ (our randomized lattice rule \emph{with} shift) is 
\[ e^{\rms}(H_{\alpha,\bsgamma}; \tilde{A}^{\ran}_{r,M}) = \sup_{\substack{f\in H_{\alpha,\bsgamma}\\ \|f\|_{\alpha,\bsgamma}\leq 1}}\left(\EE_{N,\bsz_1,\ldots,\bsz_r,\bsdelta}\left[ |I(f)-Q_{P_{N,\bsz^*}+\bsdelta}(f)|^2\right]\right)^{1/2}. \]

The following theorem provides an upper bound on $e^{\ran}(H_{\alpha,\bsgamma}; A^{\ran}_{r,M})$.
\begin{theorem}\label{thm:randomized_error}
    Let $\alpha>1/2$ be a real number and $\bsgamma\in \RR_{\geq 0}^{\NN}$ be a sequence of non-negative weights, with each $\gamma_j$ bounded above by $1$. Furthermore, let $r, M\in \NN$, with $M\geq 2$, be given. Assume that $M$ and $\eta\in (0,1)$ satisfy
    \begin{align}\label{eq:number_points}
        (1-\eta)M \geq \inf_{1/2\leq \lambda<\alpha}\, 2\prod_{j=1}^{d}\left(1+2\gamma_j^{1/\lambda}\zeta(\alpha/\lambda)\right).
    \end{align} 
    Then it holds that
    \begin{align*}
        e^{\ran}(H_{\alpha,\bsgamma}; A^{\ran}_{r,M}) & \leq \frac{rC_{\lambda,\delta}}{(1-\eta)^{\lambda-\delta-1/2}M^{\lambda-\delta+1/2}}\prod_{j=1}^{d}\left(1+2\gamma_j^{1/\lambda}\zeta(\alpha/\lambda)\right)^{\lambda-\delta}\\
        & \qquad + (1-\eta)^r \prod_{j=1}^{d}\left(1+2\gamma_j^{2}\zeta(2\alpha)\right)^{1/2}, 
        \end{align*}
        for any $1/2<\lambda<\alpha$ and $0<\delta<\min(\lambda-1/2,1)$, where $C_{\lambda,\delta}$ is a positive constant depending only on $\lambda$ and $\delta$.
\end{theorem}
\begin{proof}
    For any function $f\in H_{\alpha,\bsgamma}$, by considering its Fourier series and then applying Lemma~\ref{lem:character} and the triangle inequality, we have
    \begin{align}
        & \EE\left[ |I(f)-Q_{P_{N,\bsz^*}}(f)|\right] \notag \\
        & = \EE\left[\left|\hat{f}(\bszero)-\frac{1}{N}\sum_{\bsx\in P_{N,\bsz^*}}\sum_{\bsk\in \ZZ^d}\hat{f}(\bsk)\, \re^{2\pi \ri \bsk\cdot \bsx}\right|\right] \notag \\
        & = \EE\left[\left|\hat{f}(\bszero)-\sum_{\bsk\in P_{N,\bsz^*}^{\perp}}\hat{f}(\bsk)\right|\right] \notag \\
        & \leq \EE\left[\sum_{\bsk\in P_{N,\bsz^*}^{\perp}\setminus \{\bszero\}}|\hat{f}(\bsk)|\right] \notag \\
        & = \EE\left[\left( \chi(\bsz^*\in \Zcal_{N,\eta,\alpha,\bsgamma})+\chi(\bsz^*\not\in \Zcal_{N,\eta,\alpha,\bsgamma})\right)\sum_{\bsk\in P_{N,\bsz^*}^{\perp}\setminus \{\bszero\}}|\hat{f}(\bsk)|\right] \notag \\
        & \leq \EE\left[\chi(\bsz^*\in \Zcal_{N,\eta,\alpha,\bsgamma})\sum_{\bsk\in P_{N,\bsz^*}^{\perp}\setminus \{\bszero\}}|\hat{f}(\bsk)|\right] \notag \\
        & \qquad \qquad + \EE\left[\chi(\bsz^*\not\in \Zcal_{N,\eta,\alpha,\bsgamma})\right]\sum_{\bsk\in \ZZ^d \setminus \{\bszero\}}|\hat{f}(\bsk)|. \label{eq:bound_on_random_error_proof_1}
    \end{align}
    
    Let us focus on the first term in \eqref{eq:bound_on_random_error_proof_1}. By the definition of $\Zcal_{N,\eta,\alpha,\bsgamma}$, and by Lemma~\ref{lem:worst-case} and Proposition~\ref{prop:prob_1}, it holds for any given $N\in \Pcal_M$ and $\bsz^*\in \Zcal_{N,\eta,\alpha,\bsgamma}$ that
    \[ \sum_{\bsk\in P_{N,\bsz^*}^{\perp}\setminus \{\bszero\}}\frac{1}{(r_{\alpha,\bsgamma}(\bsk))^2}\leq B_{N,\eta,\alpha,\bsgamma,\lambda}^2 \leq B_{\lceil M/2\rceil+1,\eta,\alpha,\bsgamma,\lambda}^2 ,\]
    for all $1/2\leq \lambda<\alpha$. By defining 
    \[ H_{M,\eta,\alpha,\bsgamma}:=\inf_{1/2\leq \lambda<\alpha}\left(\frac{2}{(1-\eta)M}\prod_{j=1}^{d}\left(1+2\gamma_j^{1/\lambda}\zeta(\alpha/\lambda)\right)\right)^{\lambda}, \]
    this means that, under the condition $\bsz^*\in \Zcal_{N,\eta,\alpha,\bsgamma}$, we have $\bsk\not\in P_{N,\bsz^*}^{\perp}\setminus \{\bszero\}$ if $1/r_{\alpha,\bsgamma}(\bsk)>H_{M,\eta,\alpha,\bsgamma}$ holds. We note here that the condition \eqref{eq:number_points} ensures that $H_{M,\eta,\alpha,\bsgamma}\leq 1.$ Together with this observation, by applying Lemma~\ref{lem:prob_in_dual} and Cauchy--Schwarz inequality, we obtain
    \begin{align*}
        & \EE\left[\chi(\bsz^*\in \Zcal_{N,\eta,\alpha,\bsgamma})\sum_{\bsk\in P_{N,\bsz^*}^{\perp}\setminus \{\bszero\}}|\hat{f}(\bsk)|\right] \\
        & \leq \EE\left[\sum_{\substack{\bsk\in P_{N,\bsz^*}^{\perp}\setminus \{\bszero\}\\ r_{\alpha,\bsgamma}(\bsk)\geq 1/H_{M,\eta,\alpha,\bsgamma}}}|\hat{f}(\bsk)|\right] \\
        & = \sum_{\substack{\bsk\in \ZZ^d \setminus \{\bszero\}\\ r_{\alpha,\bsgamma}(\bsk)\geq 1/H_{M,\eta,\alpha,\bsgamma}}}\PP\left[\bsk\in P_{N,\bsz^*}^{\perp}\right]|\hat{f}(\bsk)| \\
        & \leq c\frac{r}{M}\sum_{\substack{\bsk\in \ZZ^d \setminus \{\bszero\}\\ r_{\alpha,\bsgamma}(\bsk)\geq 1/H_{M,\eta,\alpha,\bsgamma}}}|\hat{f}(\bsk)|\log(1+\|\bsk\|_{\infty}) \\
        & \leq c\frac{r}{M}\left(\sum_{\substack{\bsk\in \ZZ^d \setminus \{\bszero\}\\ r_{\alpha,\bsgamma}(\bsk)\geq 1/H_{M,\eta,\alpha,\bsgamma}}}(|\hat{f}(\bsk)|r_{\alpha,\bsgamma}(\bsk))^2\right)^{1/2}\\
        & \qquad \qquad \times \left(\sum_{\substack{\bsk\in \ZZ^d \setminus \{\bszero\}\\ r_{\alpha,\bsgamma}(\bsk)\geq 1/H_{M,\eta,\alpha,\bsgamma}}}\frac{(\log(1+\|\bsk\|_{\infty}))^2}{(r_{\alpha,\bsgamma}(\bsk))^2}\right)^{1/2}\\
        & \leq c\frac{r}{M}\|f\|_{\alpha,\bsgamma}\left(\sum_{\substack{\bsk\in \ZZ^d \setminus \{\bszero\}\\ r_{\alpha,\bsgamma}(\bsk)\geq 1/H_{M,\eta,\alpha,\bsgamma}}}\frac{(\log(1+\|\bsk\|_{\infty}))^2}{(r_{\alpha,\bsgamma}(\bsk))^2}\right)^{1/2}.
        \end{align*}
        Although we omit the remaining argument for bounding the last sum over $\bsk$ from above since it follows exactly the same way as those in \cite{dick2022component,kritzer2019lattice}, see also \cite[Chapter~11]{dick2022lattice}, we can show that, for any $1/2<\lambda<\alpha$ and $0<\delta<\min(\lambda-1/2,1)$, there exists a constant $c'_{\lambda,\delta}>0$, depending only on $\lambda$ and $\delta$, such that
        \begin{align*}
            & \sum_{\substack{\bsk\in \ZZ^d \setminus \{\bszero\}\\ r_{\alpha,\bsgamma}(\bsk)\geq 1/H_{M,\eta,\alpha,\bsgamma}}}\frac{(\log(1+\|\bsk\|_{\infty}))^2}{(r_{\alpha,\bsgamma}(\bsk))^2}\\
            & \qquad \qquad \leq c'_{\lambda,\delta}\left(H_{M,\eta,\alpha,\bsgamma}\right)^{2-(2\delta+1)/\lambda}\prod_{j=1}^{d}\left(1+2\gamma_j^{1/\lambda}\zeta(\alpha/\lambda)\right)
        \end{align*}
        holds. We note that the assumption $\gamma_j\leq 1$ for all $j$ is implicitly used here. This gives an upper bound on the first term in \eqref{eq:bound_on_random_error_proof_1} as
        \begin{align*}
            & \EE\left[\chi(\bsz^*\in \Zcal_{N,\eta,\alpha,\bsgamma})\sum_{\bsk\in P_{N,\bsz^*}^{\perp}\setminus \{\bszero\}}|\hat{f}(\bsk)|\right] \\
            & \leq c\frac{r}{M}\|f\|_{\alpha,\bsgamma}\left( c'_{\lambda,\delta}\left(H_{M,\eta,\alpha,\bsgamma}\right)^{2-(2\delta+1)/\lambda}\prod_{j=1}^{d}\left(1+2\gamma_j^{1/\lambda}\zeta(\alpha/\lambda)\right)\right)^{1/2}\\
            & \leq c \frac{r{c'_{\lambda,\delta}}^{1/2}}{M}\|f\|_{\alpha,\bsgamma}\left(\frac{2}{(1-\eta)M}\prod_{j=1}^{d}\left(1+2\gamma_j^{1/\lambda}\zeta(\alpha/\lambda)\right)\right)^{\lambda-(2\delta+1)/2}\\
            & \qquad \qquad \times \left(\prod_{j=1}^{d}\left(1+2\gamma_j^{1/\lambda}\zeta(\alpha/\lambda)\right)\right)^{1/2}\\
            & = \|f\|_{\alpha,\bsgamma}\frac{r C_{\lambda,\delta}}{(1-\eta)^{\lambda-\delta-1/2}M^{\lambda-\delta+1/2}}\prod_{j=1}^{d}\left(1+2\gamma_j^{1/\lambda}\zeta(\alpha/\lambda)\right)^{\lambda-\delta},
        \end{align*}
        for any $1/2<\lambda<\alpha$ and $0<\delta<\min(\lambda-1/2,1)$, where the constant $C_{\lambda,\delta}>0$ depends only on $\lambda$ and $\delta$.
        
        Regarding the second term in \eqref{eq:bound_on_random_error_proof_1}, Lemma~\ref{lem:prob_r} tells us
        \[ \PP\left[\bsz^*\not\in \Zcal_{N,\eta,\alpha,\bsgamma}\right]=1-\PP\left[\bsz^*\in \Zcal_{N,\eta,\alpha,\bsgamma}\right]\leq (1-\eta)^r,\]
        so that, by applying Cauchy--Schwarz inequality, we obtain
        \begin{align*}
        & \EE\left[\chi(\bsz^*\not\in \Zcal_{N,\eta,\alpha,\bsgamma})\right]\sum_{\bsk\in \ZZ^d \setminus \{\bszero\}}|\hat{f}(\bsk)| \\
        & = \PP\left[\bsz^*\not\in \Zcal_{N,\eta,\alpha,\bsgamma}\right]\sum_{\bsk\in \ZZ^d \setminus \{\bszero\}}|\hat{f}(\bsk)| \\
        & \leq (1-\eta)^r \left(\sum_{\bsk\in \ZZ^d \setminus \{\bszero\}}(|\hat{f}(\bsk)|r_{\alpha,\bsgamma}(\bsk))^2\right)^{1/2}\left(\sum_{\bsk\in \ZZ^d \setminus \{\bszero\}}\frac{1}{(r_{\alpha,\bsgamma}(\bsk))^2}\right)^{1/2} \\
        & \leq (1-\eta)^r \|f\|_{\alpha,\bsgamma}\prod_{j=1}^{d}\left(1+2\gamma_j^{2}\zeta(2\alpha)\right)^{1/2}.
        \end{align*}
        
        Thus, by taking the supremum of the randomized error over the unit ball of $H_{\alpha,\bsgamma}$, the worst-case randomized error is bounded above by
        \begin{align*}
        e^{\ran}(H_{\alpha,\bsgamma}; A^{\ran}_{r,M}) & \leq \frac{rC_{\lambda,\delta} }{(1-\eta)^{\lambda-\delta-1/2}M^{\lambda-\delta+1/2}}\prod_{j=1}^{d}\left(1+2\gamma_j^{1/\lambda}\zeta(\alpha/\lambda)\right)^{\lambda-\delta}\\
        & \qquad + (1-\eta)^r \prod_{j=1}^{d}\left(1+2\gamma_j^{2}\zeta(2\alpha)\right)^{1/2},
        \end{align*}
        for any $1/2<\lambda<\alpha$ and $0<\delta<\min(\lambda-1/2,1)$.
        This completes the proof.
\end{proof}

By choosing $r$ such that the two terms appearing in the upper bound on the worst-case randomized error are balanced, our new randomized lattice rule can achieve nearly the optimal rate of convergence of $O(M^{-\alpha-1/2})$. The proof is straightforward, so we omit it in this paper.
\begin{corollary}\label{cor:randomized_error}
    Let $M\in \NN$, with $M\geq 2$, be given and assume that $M$ and $\eta\in (0,1)$ satisfy \eqref{eq:number_points}. Let 
    \[ r=\left\lceil -\left(\alpha+\frac{1}{2}\right)\frac{\log M}{\log (1-\eta)}\right\rceil.\]
    Then $e^{\ran}(H_{\alpha,\bsgamma}; A^{\ran}_{r,M})$ decays with the order $M^{-\lambda+\delta-1/2}\log M$ for any $1/2<\lambda<\alpha$ and $0<\delta<\min(\lambda-1/2,1)$, which is arbitrarily close to $M^{-\alpha-1/2}$ when $\lambda\to \alpha^{-}$ and $\delta\to 0^{+}$. Moreover, $e^{\ran}(H_{\alpha,\bsgamma}; A^{\ran}_{r,M})$ is bounded independently of the dimension $d$ if $\sum_{j=1}^{\infty}\gamma_j^{1/\alpha}<\infty$.
\end{corollary}

The result shown in Theorem~\ref{thm:randomized_error} is now extended to $e^{\rms}(H_{\alpha,\bsgamma}; \tilde{A}^{\ran}_{r,M})$ as follows:
\begin{theorem}\label{thm:rmse}
    Let $\alpha>1/2$ be a real number and $\bsgamma\in \RR_{\geq 0}^{\NN}$ be a sequence of non-negative weights, with each $\gamma_j$ bounded above by $1$. Furthermore, let $r,M\in \NN$, with $M\geq 2$, be given. Assume that $M$ and $\eta\in (0,1)$ satisfy \eqref{eq:number_points}. 
    Then it holds that
    \begin{align*}
    & \left(e^{\rms}(H_{\alpha,\bsgamma}; \tilde{A}^{\ran}_{r,M})\right)^2 \\
    & \leq \frac{rC_{\lambda,\delta}}{\alpha (1-\eta)^{\lambda(2-\delta)}M^{\lambda(2-\delta)+1}}\prod_{j=1}^{d}\left(1+2\gamma_j^{1/\lambda}\zeta(\alpha/\lambda)\right)^{\lambda(2-\delta)}+(1-\eta)^r,
    \end{align*}
        for any $1/2\leq \lambda<\alpha$ and $0<\delta\leq 1/\alpha$, where $C_{\lambda,\delta}$ is a positive constant depending only on $\lambda$ and $\delta$.
\end{theorem}
\begin{proof}
As considered in the proof of Theorem~\ref{thm:randomized_error}, the Fourier series of $f\in H_{\alpha,\bsgamma}$, Lemma~\ref{lem:character} and the triangle inequality lead to
\begin{align}
    & \EE\left[ |I(f)-Q_{P_{N,\bsz^*}+\bsdelta}(f)|^2\right] \notag \\
    & = \EE\left[\left|\hat{f}(\bszero)-\frac{1}{N}\sum_{\bsx\in P_{N,\bsz^*}}\sum_{\bsk\in \ZZ^d}\hat{f}(\bsk)\, \re^{2\pi \ri \bsk\cdot (\bsx+\bsdelta)}\right|\right] \notag \\
    & = \EE\left[ \left|\sum_{\bsk\in P_{N,\bsz^*}^{\perp}\setminus \{\bszero\}}\hat{f}(\bsk)\, \re^{2\pi \ri \bsk\cdot \bsdelta}\right|^2\right] \notag \\
    & = \EE\left[ \sum_{\bsk,\bsell\in P_{N,\bsz^*}^{\perp}\setminus \{\bszero\}}\hat{f}(\bsk)\overline{\hat{f}(\bsell)}\, \re^{2\pi \ri (\bsk-\bsell)\cdot \bsdelta}\right] \notag \\
    & = \EE_{N,\bsz_1,\ldots,\bsz_r}\left[ \sum_{\bsk,\bsell\in P_{N,\bsz^*}^{\perp}\setminus \{\bszero\}}\hat{f}(\bsk)\overline{\hat{f}(\bsell)}\, \EE_{\bsdelta}\left[\re^{2\pi \ri (\bsk-\bsell)\cdot \bsdelta}\right]\right] \notag \\
    & = \EE_{N,\bsz_1,\ldots,\bsz_r}\left[ \sum_{\bsk\in P_{N,\bsz^*}^{\perp}\setminus \{\bszero\}}|\hat{f}(\bsk)|^2\right] \notag \\
    & = \EE_{N,\bsz_1,\ldots,\bsz_r}\left[\left( \chi(\bsz^*\in \Zcal_{N,\eta,\alpha,\bsgamma})+\chi(\bsz^*\not\in \Zcal_{N,\eta,\alpha,\bsgamma})\right)\sum_{\bsk\in P_{N,\bsz^*}^{\perp}\setminus \{\bszero\}}|\hat{f}(\bsk)|^2\right] \notag \\
    & \leq \EE_{N,\bsz_1,\ldots,\bsz_r}\left[\chi(\bsz^*\in \Zcal_{N,\eta,\alpha,\bsgamma})\sum_{\bsk\in P_{N,\bsz^*}^{\perp}\setminus \{\bszero\}}|\hat{f}(\bsk)|^2\right] \notag \\
    & \qquad + \PP_{N,\bsz_1,\ldots,\bsz_r}\left[\bsz^*\not\in \Zcal_{N,\eta,\alpha,\bsgamma}\right]\sum_{\bsk\in \ZZ^d \setminus \{\bszero\}}|\hat{f}(\bsk)|^2.
    \label{eq:bound_on_rmse_proof_1}
\end{align}

Following a similar argument to that made in the proof of Theorem~\ref{thm:randomized_error}, the first term in \eqref{eq:bound_on_rmse_proof_1} is bounded above by
\begin{align*}
    & \EE_{N,\bsz_1,\ldots,\bsz_r}\left[\chi(\bsz^*\in \Zcal_{N,\eta,\alpha,\bsgamma})\sum_{\bsk\in P_{N,\bsz^*}^{\perp}\setminus \{\bszero\}}|\hat{f}(\bsk)|^2\right] \\
    & \leq \EE\left[\sum_{\substack{\bsk\in P_{N,\bsz^*}^{\perp}\setminus \{\bszero\}\\ r_{\alpha,\bsgamma}(\bsk)\geq 1/H_{M,\eta,\alpha,\bsgamma}}}|\hat{f}(\bsk)|^2\right] \\
    & = \sum_{\substack{\bsk\in \ZZ^d \setminus \{\bszero\}\\ r_{\alpha,\bsgamma}(\bsk)\geq 1/H_{M,\eta,\alpha,\bsgamma}}}\PP\left[\bsk\in P_{N,\bsz^*}^{\perp}\right]|\hat{f}(\bsk)|^2\\
    & \leq c\frac{r}{M}\sum_{\substack{\bsk\in \ZZ^d \setminus \{\bszero\}\\ r_{\alpha,\bsgamma}(\bsk)\geq 1/H_{M,\eta,\alpha,\bsgamma}}}|\hat{f}(\bsk)|^2\log(1+\|\bsk\|_{\infty}) \\
    & \leq c\frac{r}{M}\left(\sum_{\substack{\bsk\in \ZZ^d \setminus \{\bszero\}\\ r_{\alpha,\bsgamma}(\bsk)\geq 1/H_{M,\eta,\alpha,\bsgamma}}}(|\hat{f}(\bsk)|r_{\alpha,\bsgamma}(\bsk))^2\right)\times \sup_{\substack{\bsk\in \ZZ^d \setminus \{\bszero\}\\ r_{\alpha,\bsgamma}(\bsk)\geq 1/H_{M,\eta,\alpha,\bsgamma}}}\frac{\log(1+\|\bsk\|_{\infty})}{(r_{\alpha,\bsgamma}(\bsk))^2}\\
    & \leq c\frac{r}{M}\|f\|_{\alpha,\bsgamma}^2\sup_{\substack{\bsk\in \ZZ^d \setminus \{\bszero\}\\ r_{\alpha,\bsgamma}(\bsk)\geq 1/H_{M,\eta,\alpha,\bsgamma}}}\frac{\log(1+\|\bsk\|_{\infty})}{(r_{\alpha,\bsgamma}(\bsk))^2}.
\end{align*}
As done in \cite[Theorem~11]{kritzer2019lattice}, using the elementary inequality $\log(1+x)\leq x^{\alpha \delta}/(\alpha \delta)$ for all $\delta\in (0,1/\alpha]$ along with the assumption $\gamma_j\leq 1$ for all $j$, we can derive
\begin{align*}
    \sup_{\substack{\bsk\in \ZZ^d \setminus \{\bszero\}\\ r_{\alpha,\bsgamma}(\bsk)\geq 1/H_{M,\eta,\alpha,\bsgamma}}}\frac{\log(1+\|\bsk\|_{\infty})}{(r_{\alpha,\bsgamma}(\bsk))^2} & \leq \frac{1}{\alpha \delta}\sup_{\substack{\bsk\in \ZZ^d \setminus \{\bszero\}\\ r_{\alpha,\bsgamma}(\bsk)\geq 1/H_{M,\eta,\alpha,\bsgamma}}}\frac{\|\bsk\|_{\infty}^{\alpha \delta}}{(r_{\alpha,\bsgamma}(\bsk))^2}\\
    & \leq \frac{1}{\alpha \delta}\sup_{\substack{\bsk\in \ZZ^d \setminus \{\bszero\}\\ r_{\alpha,\bsgamma}(\bsk)\geq 1/H_{M,\eta,\alpha,\bsgamma}}}\frac{1}{(r_{\alpha,\bsgamma}(\bsk))^{2-\delta}}\\
    & \leq \frac{1}{\alpha \delta}\left( H_{M,\eta,\alpha,\bsgamma}\right)^{2-\delta}.
\end{align*}

The second term in \eqref{eq:bound_on_rmse_proof_1} can also be bounded above in a similar manner to the proof of Theorem~\ref{thm:randomized_error} as follows:
\begin{align*}
    & \PP_{N,\bsz_1,\ldots,\bsz_r}\left[\bsz^*\not\in \Zcal_{N,\eta,\alpha,\bsgamma}\right]\sum_{\bsk\in \ZZ^d \setminus \{\bszero\}}|\hat{f}(\bsk)|^2 \\
    & \leq  (1-\eta)^r \sum_{\bsk\in \ZZ^d \setminus \{\bszero\}}|\hat{f}(\bsk)|^2\\
    & \leq (1-\eta)^r \left(\sum_{\bsk\in \ZZ^d \setminus \{\bszero\}}(|\hat{f}(\bsk)|r_{\alpha,\bsgamma}(\bsk))^2\right)\sup_{\bsk\in \ZZ^d \setminus \{\bszero\}}\frac{1}{(r_{\alpha,\bsgamma}(\bsk))^2}\leq (1-\eta)^r\|f\|_{\alpha,\bsgamma}^2,
\end{align*}
where we have used the assumption $\gamma_j\leq 1$ to obtain the last inequality.

Altogether, by taking the supremum of the randomized error over the unit ball of $H_{\alpha,\bsgamma}$, we obtain
\begin{align*}
    & \left(e^{\rms}(H_{\alpha,\bsgamma}; \tilde{A}^{\ran}_{r,M})\right)^2 \\
    & \leq  c\frac{r}{\alpha \delta M}\left( H_{M,\eta,\alpha,\bsgamma}\right)^{2-\delta}+(1-\eta)^r\\
    & \leq \frac{rC_{\lambda,\delta}}{\alpha (1-\eta)^{\lambda(2-\delta)}M^{\lambda(2-\delta)+1}}\prod_{j=1}^{d}\left(1+2\gamma_j^{1/\lambda}\zeta(\alpha/\lambda)\right)^{\lambda(2-\delta)}+(1-\eta)^r,
\end{align*}
for any $1/2\leq \lambda<\alpha$ and $0<\delta\leq 1/\alpha$, where $C_{\lambda,\delta}>0$ is a constant depending only on $\lambda$ and $\delta$. This completes the proof.
\end{proof}

Similarly to the randomized error, the following corollary from Theorem~\ref{thm:rmse} demonstrates that our randomized lattice rule with shift can achieve nearly the optimal rate of the RMSE. Although we do not provide the proof again, it is worth noting a slight increase in $r$ in this case for a necessity to further reduce the failure probability.

\begin{corollary}\label{cor:rmse}
    Let $M\in \NN$, with $M\geq 2$, be given and assume that $M$ and $\eta\in (0,1)$ satisfy \eqref{eq:number_points}. Let 
    \begin{align}\label{eq:number_trials}
        r=\left\lceil -\left(2\alpha+1\right)\frac{\log M}{\log (1-\eta)}\right\rceil.
    \end{align}
    Then $e^{\rms}(H_{\alpha,\bsgamma}; \tilde{A}^{\ran}_{r,M})$ decays with the order $M^{-\lambda-1/2-\lambda\delta/2}\log M$ for any $1/2\leq \lambda<\alpha$ and $0<\delta\leq 1/\alpha$, which is arbitrarily close to $M^{-\alpha-1/2}$ when $\lambda\to \alpha^{-}$ and $\delta\to 0^{+}$. Moreover, $e^{\rms}(H_{\alpha,\bsgamma}; \tilde{A}^{\ran}_{r,M})$ is bounded independently of the dimension $d$ if $\sum_{j=1}^{\infty}\gamma_j^{1/\alpha}<\infty$.
\end{corollary}

\subsection{Additional remarks}\label{subsec:remarks}
We provide comments on our new randomized lattice rule, covering aspects such as algorithmic formulation, computational cost, and potential extensions.
\begin{remark}
    In Algorithm~\ref{alg:without}, we first randomly draw a single $N$, followed by $\bsz_1,\dots,\bsz_r$ conditioned on $N$. Even if this process were replaced by independently and randomly drawing $N_1,\ldots,N_r$ from the set $\Pcal_M$ first, and then each $\bsz_j$ conditioned on $N_j$, a randomized lattice rule based on the pair $(N^*,\bsz^*)$ that minimizes the worst-case error $e^{\wor}(H_{\alpha,\bsgamma}; P_{N_j,\bsz_j})$ among $(N_1,\bsz_1),\ldots,(N_r,\bsz_r)$ can still yield similar upper bounds on both the (worst-case) randomized error and the RMSE.
\end{remark}

\begin{remark}
    As mentioned in Lemma~\ref{lem:worst-case}, computing $e^{\wor}(H_{\alpha,\bsgamma}; P_{N,\bsz_j})$ requires a computational cost of $O(dM)$ for each $j\in \{1,\ldots,r\}$. Thus, the necessary cost to run Algorithm~\ref{alg:without} once is $O(rdM)$. To achieve nearly the optimal rate of the two error criteria with our new randomized rule, it suffices to set 
    \[ r=\left\lceil -\left(\alpha+\frac{1}{2}\right)\frac{\log M}{\log (1-\eta)}\right\rceil\quad \text{or}\quad r=\left\lceil -\left(2\alpha+1\right)\frac{\log M}{\log (1-\eta)}\right\rceil, \]
    as discussed in Corollaries~\ref{cor:randomized_error} and \ref{cor:rmse}, respectively. This results in a cost of $O(\alpha dM\log M)$, which has an additional factor $\alpha$ compared to the cost of the randomized CBC algorithm in \cite{dick2022component}. However, by exploiting parallel computation for computing $e^{\wor}(H_{\alpha,\bsgamma}; P_{N,\bsz_j})$ for all $j$, the overall cost of running Algorithm~\ref{alg:without} once remains $O(dM)$, if $O(\log M)$ computational nodes are available. It should be noted that the randomized CBC algorithm in \cite{dick2022component} requires a cost of $O(dM\log M)$, using the fast Fourier transform. The $\log M$ factor in the CBC algorithm is not easily parallelizable, which makes our approach slightly more efficient in a parallel computation setting.
\end{remark}

\begin{remark}\label{rem:stable}
The results obtained thus far have been built upon the situation where the parameters $\alpha$ and $\bsgamma$ are well specified. To demonstrate the stability of our randomized lattice rule, let us consider a different parameter set $\beta$ and $\bsgamma'=(\gamma'_j)_{j\in \NN}$ such that
\[ \beta>\alpha>\frac{1}{2}\quad \text{and} \quad \gamma'_j=\gamma_j^{\beta/\alpha}\quad \text{for all $j$.}\]
Since the $\ell_p$-norm for an infinite sequence is decreasing in $p\geq 1$, we can infer
\begin{align*}
    \left(e^{\wor}(H_{\beta,\bsgamma'}; Q_{P_{N,\bsz}})\right)^2 & = \sum_{\bsk\in P_{N,\bsz}^{\perp}\setminus \{\bszero\}}\frac{1}{(r_{\beta,\bsgamma'}(\bsk))^2} = \sum_{\bsk\in P_{N,\bsz}^{\perp}\setminus \{\bszero\}}\frac{1}{(r_{\alpha,\bsgamma}(\bsk))^{2\beta/\alpha}} \\
    & \leq \left(\sum_{\bsk\in P_{N,\bsz}^{\perp}\setminus \{\bszero\}}\frac{1}{(r_{\alpha,\bsgamma}(\bsk))^{2}}\right)^{\beta/\alpha} = \left(e^{\wor}(H_{\alpha,\bsgamma}; Q_{P_{N,\bsz}})\right)^{2\beta/\alpha}.
\end{align*}
This suggests that our randomized lattice rule can adapt to the weighted Korobov space with higher smoothness $\beta$. To ensure that the second term in the upper bounds given in Theorems~\ref{thm:randomized_error} and \ref{thm:rmse} is not dominant, it suffices to set
\begin{align}\label{eq:stable_r}
r=\left\lceil -g(M)\frac{\log M}{\log (1-\eta)}\right\rceil,
\end{align}
where $g: \NN\to \RR_{\geq 0}$ is a slowly increasing function such as $g(M)=\log M$ or $g(M)=\max(\log\log M,1)$. We refer to \cite{dick2021stability} for further results on the stability of rank-1 lattice rules in the deterministic setting.
\end{remark}

\begin{remark}
    Our results obtained so far can be naturally extended to the weighted half-period cosine spaces \cite{cools2016tent,dick2014lattice}, which includes the weighted unanchored Sobolev space with first-order dominating mixed smoothness as a special case. This extension can be attained by applying the tent transformation
    \[ \pi(x)=1-|2x-1|\]
    componentwise to every point in $P_{N,\bsz^*}$ or $P_{N,\bsz^*}+\bsdelta$, as described in \cite[Section~4]{dick2022component}. However, whether the same extension applies to the weighted unanchored Sobolev space with second-order dominating mixed smoothness remains an open question for future research. We refer to \cite{goda2019lattice,hickernell2002obtaining} as relevant literature in this context.
    
    Similarly, our findings can be extended to the weighted Walsh spaces by replacing rank-1 lattice point sets with (infinite-precision) rank-1 polynomial lattice point sets. Here, instead of randomly choosing $N$ (the number of points), we randomly draw an irreducible polynomial $p$ with a fixed degree, as investigated in \cite[Section~5]{dick2022component}.
\end{remark}

\begin{remark}
    In \cite{goda2022note}, the author studied the concatenation of rank-1 lattice points in dimension $d$ with random points in dimension $s-d\, (>0)$ to approximate $s$-variate integrals. This approach is also applicable to our randomized lattice rule. Although we omit the details, under the assumption $\gamma_j\leq 1$ for all $j$, the squared worst-case RMSE of such a concatenated randomized cubature rule with a random shift can be shown to be bounded above by the addition of the two terms shown in Theorem~\ref{thm:rmse} and an additional term
    \[ \frac{2}{M}\max_{j=d+1,\ldots,s}\gamma_j^2.\]
    If the weights are arranged in decreasing order, i.e., $\gamma_1\geq \gamma_2\geq \cdots\geq 0$, and satisfy the summability condition $\sum_{j=1}^{\infty}\gamma_j^{1/\alpha}<\infty$, we have
    \[ \infty>\sum_{j=1}^{\infty}\gamma_j^{1/\alpha}\geq \sum_{j=1}^{d}\gamma_j^{1/\alpha}\geq d\gamma_d^{1/\alpha}, \]
    which implies the existence of constants $C>0$ and $\beta\geq \alpha$ such that
    \[ \max_{j=d+1,\ldots,s}\gamma_j^2=\gamma_{d+1}^2\leq C(d+1)^{-\beta},\] 
    for any $d\in \NN$. 
    By choosing $d= c \cdot \lceil M^{\alpha/\beta}\rceil$ for some positive constant $c$, the additional term arising from the concatenation of random points does not dominate, ensuring that the rate of convergence remains nearly optimal, regardless of how large $s$ is.
\end{remark}

\section{Numerical experiments}\label{sec:numerics}
In this final section, we present two types of numerical experiments. In Section~\ref{subsec:numerics1}, we compare the distributions of the worst-case errors of rank-1 lattice rules obtained by our Algorithm~\ref{alg:without} with those obtained by the randomized CBC construction \cite{dick2022component}. Following this, in Section~\ref{subsec:numerics2}, we assess the empirical performance of these two randomized lattice rules by applying them to test functions.

\subsection{Distribution of the worst-case error}\label{subsec:numerics1}
First, we empirically investigate the distribution of the worst-case error in the same weighted Korobov space using two different randomized lattice rules: our new algorithm (Algorithm~\ref{alg:without}) and the randomized CBC construction introduced by \cite{dick2022component}. Our new algorithm selects a generating vector by randomly drawing generating vectors multiple times from all possible candidates and choosing the one that minimizes the worst-case error among them. This ensures that the worst-case error achieves nearly the optimal order with a very high probability. In fact, as shown in Lemma~\ref{lem:prob_r}, the probability of the event $\bsz^*\in \Zcal_{N,\eta,\alpha,\bsgamma}$ is at least $1-(1-\eta)^r$. On the other hand, any generating vector obtained from the randomized CBC construction guarantees that the worst-case error of the corresponding lattice rule achieves nearly the optimal order with probability 1. Here, we aim to compare the worst-case error distributions of rank-1 lattice rules obtained from these two algorithms.

\begin{figure}[p]
\begin{subfigure}{.48\textwidth}
  \centering
  \includegraphics[width=\linewidth]{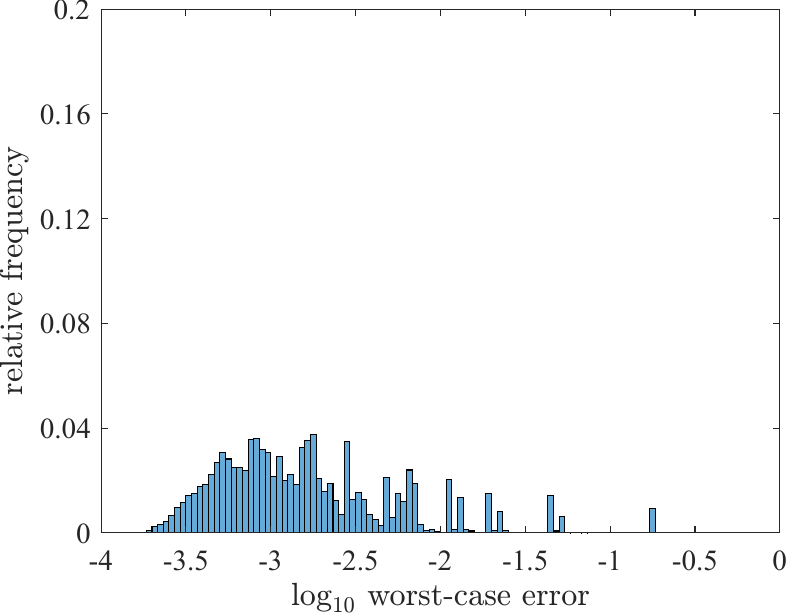}
\end{subfigure}
\hfill
\begin{subfigure}{.48\textwidth}
  \centering
  \includegraphics[width=\linewidth]{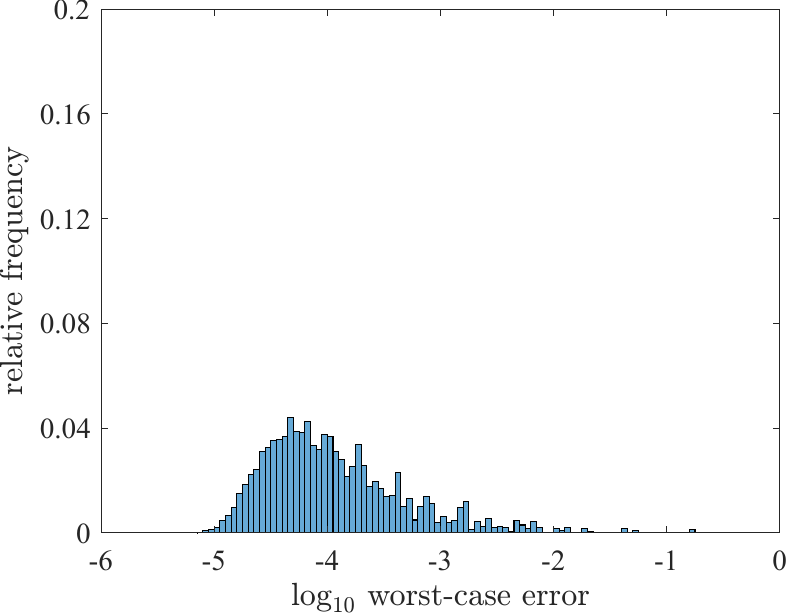}
\end{subfigure}\\[1em]
\begin{subfigure}{.48\textwidth}
  \centering
  \includegraphics[width=\linewidth]{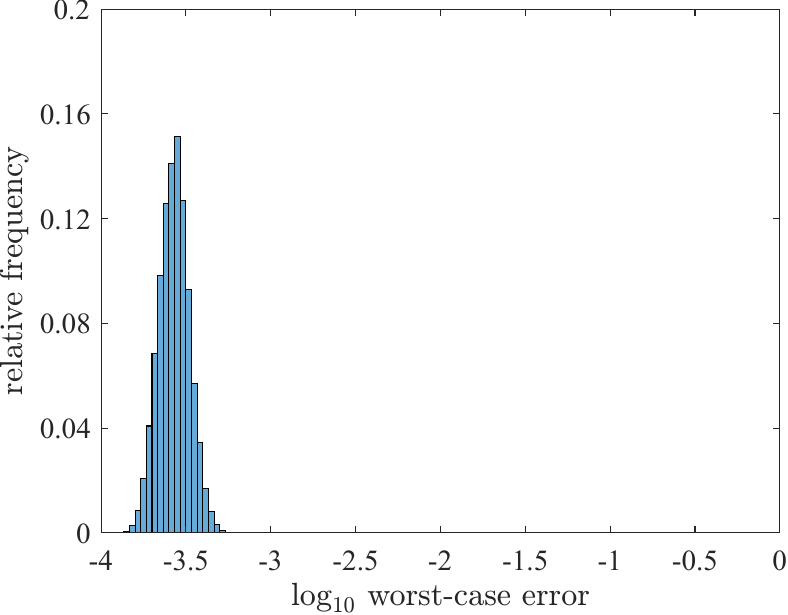}
\end{subfigure}
\hfill
\begin{subfigure}{.48\textwidth}
  \centering
  \includegraphics[width=\linewidth]{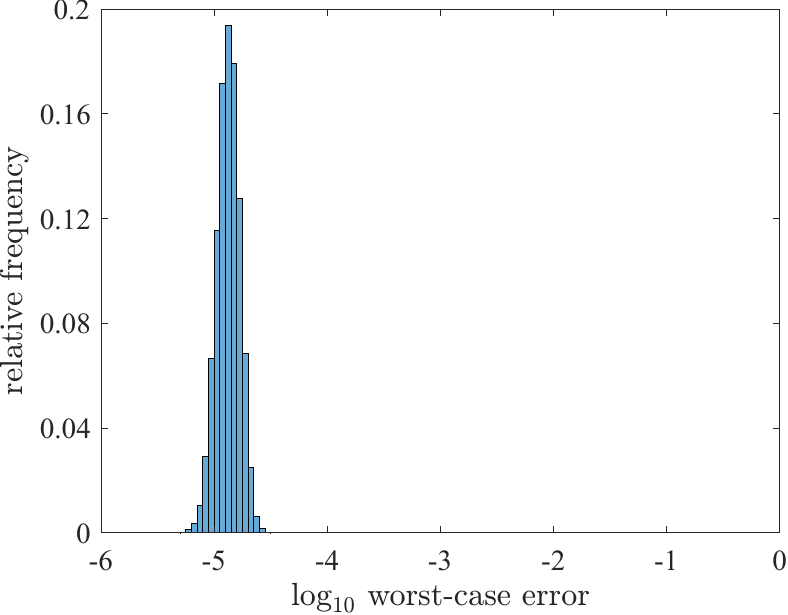}
\end{subfigure}\\[1em]
\begin{subfigure}{.48\textwidth}
  \centering
  \includegraphics[width=\linewidth]{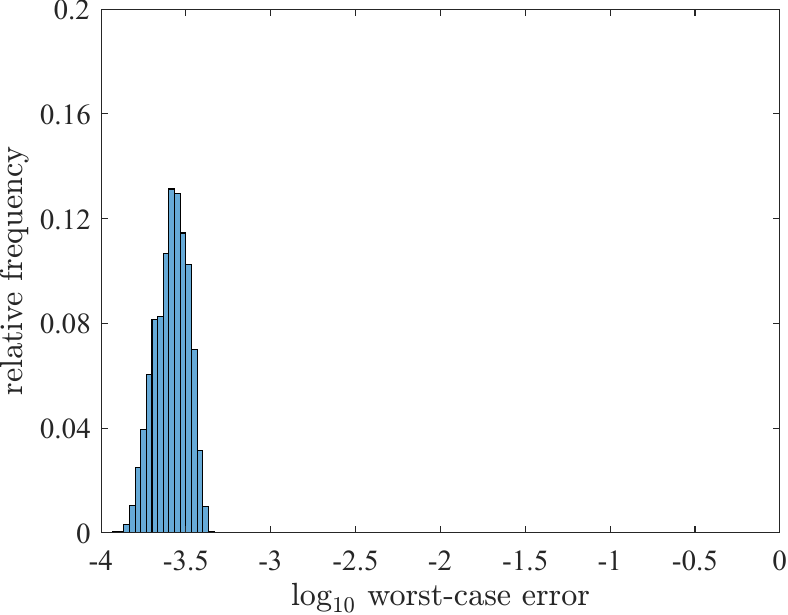}
\end{subfigure}
\hfill
\begin{subfigure}{.48\textwidth}
  \centering
  \includegraphics[width=\linewidth]{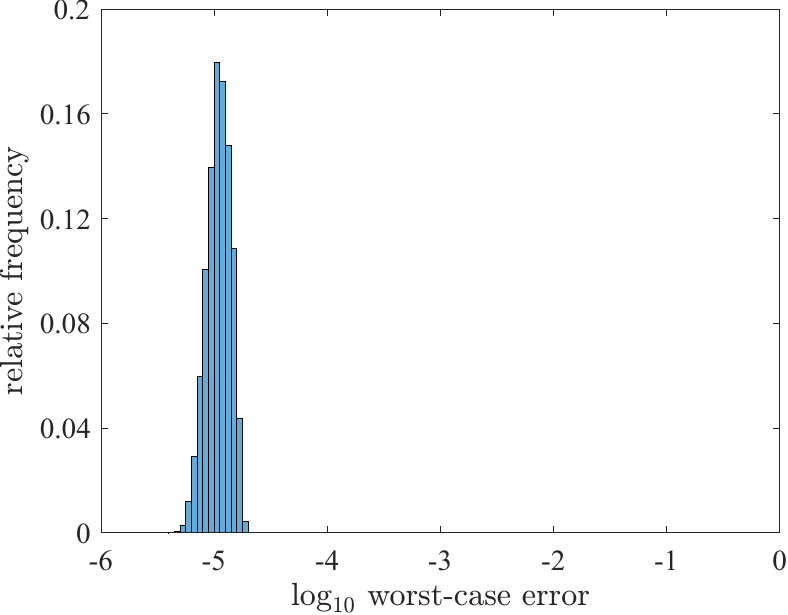}
\end{subfigure}
\caption{Relative frequency histograms of the base-10 logarithm of the worst-case error in the weighted Korobov space with $d=20$, $\alpha=2$, and $\gamma_j=1/j^3$. The left column displays the results for $N=251$, while the right column shows those for $N=2039$. The top row displays the results for rank-1 lattice points with randomly chosen $\bsz\in \{1,\ldots,N-1\}^d$, the middle row for those obtained by Algorithm~\ref{alg:without} with fixed $N$, $\eta=1/2$, and $r$ given by \eqref{eq:number_trials} in which $M$ is replaced by $N$, and the bottom row for those generated by the randomized CBC construction from \cite{dick2022component} with fixed $N$ and $\eta=1/2$ (note that the symbol $\tau$ was used instead of $\eta$ in \cite{dick2022component}).}
\label{fig:distribution}
\end{figure}

Consider the weighted Korobov space $H_{\alpha,\bsgamma}$ with $\alpha=2$ and $\gamma_j=1/j^3$ in dimension $d=20$. As a reference, the top two panels of Figure~\ref{fig:distribution} display the empirical distributions of the worst-case error  (logarithmically scaled with base-$10$) in this space for rank-1 lattice rules when the generating vectors are randomly chosen from $\{1,\ldots,N-1\}^d$ for two different primes $N=251$ (left) and $N=2039$ (right), respectively. Each distribution represents a relative frequency histogram based on $10^4$ independent and random draws of generating vectors. As evident, they exhibit right-skewed patterns, indicating that while the majority of possible generating vectors lead to small worst-case errors, there is still a non-negligible proportion of generating vectors with significantly larger worst-case errors.  A similar empirical result regarding the wide spread of worst-case error for randomly chosen generating vectors can be found in \cite[Section~4]{goda2022construction}.

In a similar experimental setup, we present the results for our new algorithm and the randomized CBC construction in the middle and bottom panels of Figure~\ref{fig:distribution}, respectively. Since it is necessary for our new algorithm to set the number of repetitions, denoted by $r$, we used the value obtained by replacing $M$ with $N$ in \eqref{eq:number_trials} with $\eta=1/2$ in this experiment. Similarly, for the randomized CBC construction described in \cite[Algorithm~2.4]{dick2022component}, we set $\tau=1/2$, where the symbol $\tau$ was used instead of $\eta$ with the same meaning. For both values of $N$, the empirical distributions obtained by these two algorithms exhibit single high peaks in the region of small worst-case errors with more symmetric patterns compared to the results in the top panels. No clear distinction is observed visually between these two algorithms. This result indicates that our new algorithm was successful in filtering out the bad generating vectors through the selection process. 

\subsection{Performance for test functions}\label{subsec:numerics2}
We apply the two randomized lattice rules, both with a random shift, by drawing the number of points and the generating vectors randomly according to either our Algorithm~\ref{alg:without} or the randomized CBC construction, to several test functions. Although the smoothness of the integrands is not always known a priori, selecting generating vectors based on the worst-case error for the weighted Korobov space with small $\alpha$ can adapt to cases with larger smoothness, as mentioned in Remark~\ref{rem:stable}. Therefore, in this experiment, we set $\alpha=1$ and $\gamma_j=1/j^2$ as inputs for both approaches. Additionally, we determine the number of repetitions $r$ using \eqref{eq:stable_r} with $g(M)=\max(\log\log M,1)$. While we fixed $N$ in the previous subsection, here we fix $M$ and randomly draw $N$ from $\Pcal_M$ as described in the first step of Algorithm~\ref{alg:without}. We use an unbiased sample variance from $50$ independent replications as a quality measure.

\begin{figure}[t]
\begin{subfigure}{.48\textwidth}
  \centering
  \includegraphics[width=\linewidth]{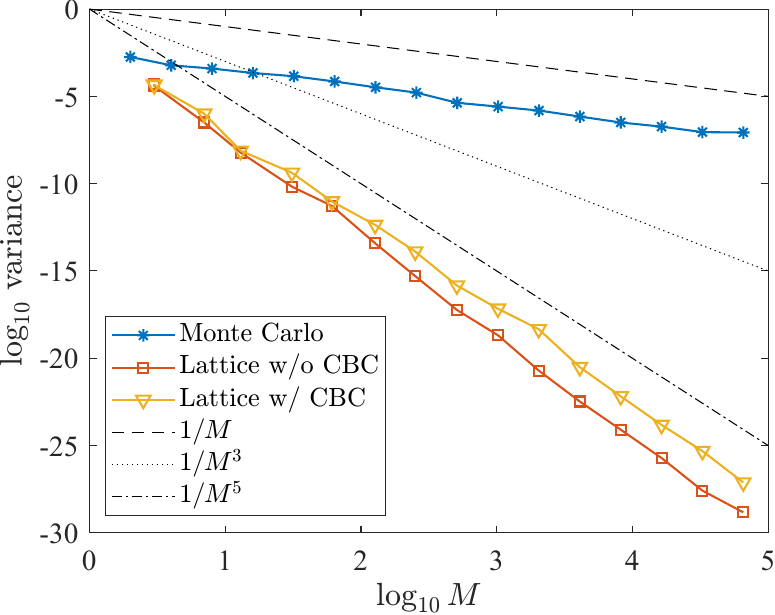}
\end{subfigure}
\hfill
\begin{subfigure}{.48\textwidth}
  \centering
  \includegraphics[width=\linewidth]{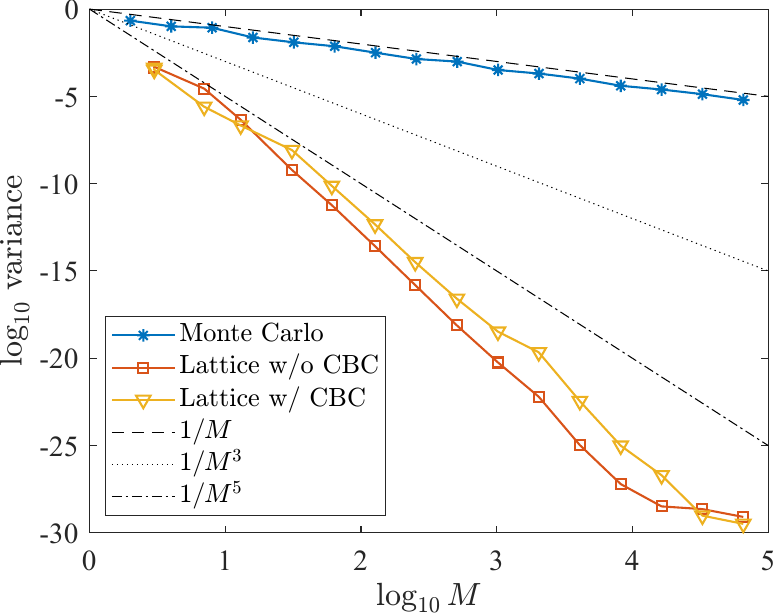}
\end{subfigure}\\[1em]
\begin{subfigure}{.48\textwidth}
  \centering
  \includegraphics[width=\linewidth]{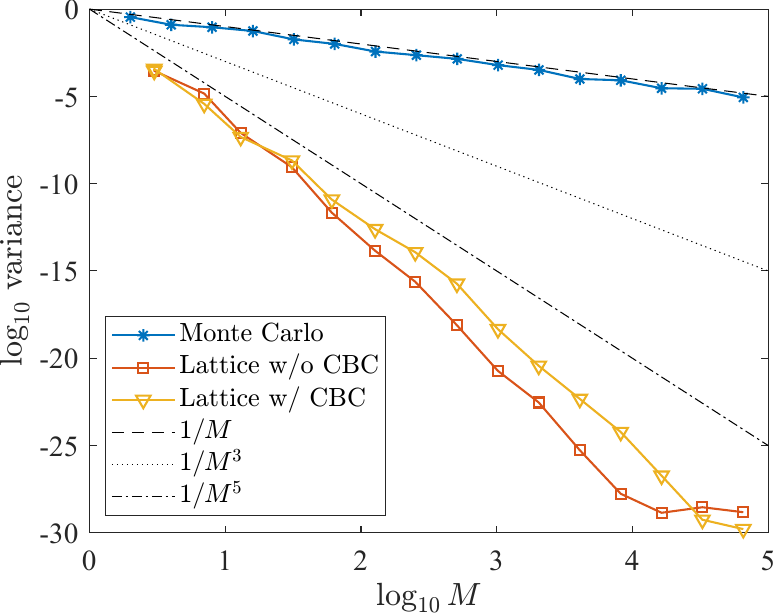}
\end{subfigure}
\hfill
\begin{subfigure}{.48\textwidth}
  \centering
  \includegraphics[width=\linewidth]{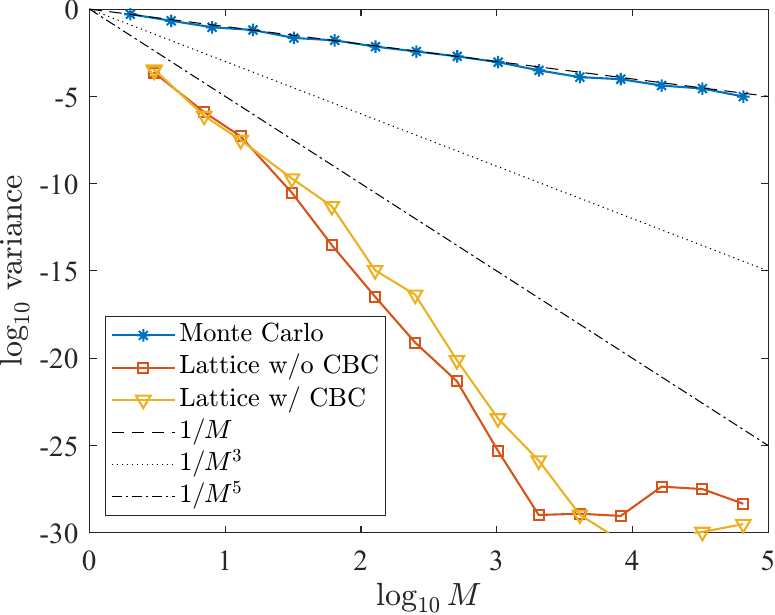}
\end{subfigure}
\caption{Convergence behavior of randomized lattice rules for test functions in $d=2$. Each panel corresponds to a different test function: $f_1$ (top left), $f_2$ (top right), $f_3$ (bottom left), and $f_4$ (bottom right). The horizontal axis represents the maximum number of points $M$, while the vertical axis does the sample variance (logarithmically scaled with base-$10$). The results of the standard Monte Carlo method are represented by blue $\ast$ for reference. The result of Algorithm~\ref{alg:without} is represented by orange $\square$, while that of the randomized CBC approach is by yellow $\triangledown$.}
\label{fig:test_d2}
\end{figure}

Our test functions in a general dimension $d$ are given as follows:
\[ f_1(\bsx)=\prod_{j=1}^{d}\left[ 1+\frac{1}{j^4}\left(x_j-\frac{1}{2}\right)^2\sin\left(2\pi x_j-\pi\right) \right]\]
and
\[ f_\beta(\bsx)=\prod_{j=1}^{d}\left[ 1+\frac{1}{j^{2\beta}}\left((2\beta+1)\binom{2\beta}{\beta}x_j^{\beta}(1-x_j)^{\beta}-1\right) \right],\]
with $\beta=2,3,4$. It can be verified that $f_1\in H_{5/2-\varepsilon,\bsgamma}$ and $f_{\beta}\in H_{(2\beta+1)/2-\varepsilon,\bsgamma}$ for any $\beta=2,3,4$ with arbitrarily small $\varepsilon>0$. This implies that the best possible convergence rates of the variance are $M^{-6+\varepsilon}, M^{-6+\varepsilon}, M^{-8+\varepsilon}, M^{-10+\varepsilon}$, respectively.
Note that the true integral for all these test functions is equal to $1$.

We present the results for all the test functions in the lower-dimensional case of $d=2$ in Figure~\ref{fig:test_d2}. The horizontal axis represents the maximum number of points $M$, while the vertical axis displays the sample variance (both on a base-$10$ logarithmic scale). Additionally, we include the results for the standard Monte Carlo method with $M$ random points as a reference. As anticipated, the variance of the standard Monte Carlo method decreases at a rate of $O(M^{-1})$. In contrast, the two randomized lattice rules demonstrate much faster convergence behaviors. Due to machine precision limitations, the variance no longer decreases beyond a certain point for $f_3$ and $f_4$. Also, we see that our new randomized rule exhibits slightly better performance than the randomized CBC approach overall.
The linear regression for the results in the suitable ranges of $M$ indicates that our proposed randomized lattice rule achieves variance rates of $M^{-5.67}, M^{-7.30}, M^{-7.04}, M^{-9.16}$, respectively.
This observation almost agrees with our theoretical findings. It is noteworthy that while we theoretically expect better convergence for $f_3$ compared to $f_2$, the experiment shows the opposite result. Investigating the reason for this discrepancy is beyond the scope of this work.

As illustrated in Figure~\ref{fig:test_d20}, high-order convergence behaviors of both the randomized lattice rules are maintained even in the high-dimensional setting with $d=20$. However, the rate of convergence deteriorates considerably; the linear regression suggests that our proposed randomized lattice rule achieves variance rates of $M^{-4.21}, M^{-4.09}, M^{-5.46}, M^{-6.77}$, respectively. Although we omit the details here, the rate of convergence can be significantly improved, for example, by replacing the weights $1/j^{2\beta}$ in $f_{\beta}$ for $\beta=2, 3, 4$ with faster-decaying weights. This confirms the importance of weight parameters in high-dimensional settings. In our current experiments, although the randomized CBC approach shows slightly superior performance for $f_2$ compared to our new algorithm, the performance of both approaches is quite comparable for other functions. Hence, in the randomized setting, lattice rules based on the CBC construction are not the only viable choice, suggesting potential for the development of alternative implementable approaches with both theoretical and empirical support.

Finally, we briefly compare our results with those of a deterministic CBC-constructed rank-1 lattice rule. Numerical experiments for the same test functions showed that, in most cases, the (squared) absolute error decay rate of the deterministic rule was comparable to, or even better than, the variance decay rate achieved by the randomized rank-1 lattice rules. This can be attributed to the following: considering a Banach space with the norm
\[ \|f\|_{\alpha,\bsgamma}:= \sup_{\bsk\in \ZZ^d}|\hat{f}(\bsk)|\, r_{\alpha,\bsgamma}(\bsk), \]
our test functions $f_1,\ldots,f_4$ belong to this space with smoothness parameters $\alpha=3,3,4,5$, respectively.
In such spaces, a deterministic rank-1 lattice rule can achieve a worst-case error of order $N^{-\alpha+\varepsilon}$ for arbitrarily small $\varepsilon>0$. 
Moreover, improving this rate by considering randomized error criteria and randomized rank-1 lattice rules seems quite difficult. Therefore, determining the superiority of deterministic versus randomized algorithms for specific functions, in terms of their convergence rates, is not straightforward and requires further investigation.

\begin{figure}[t]
\begin{subfigure}{.48\textwidth}
  \centering
  \includegraphics[width=\linewidth]{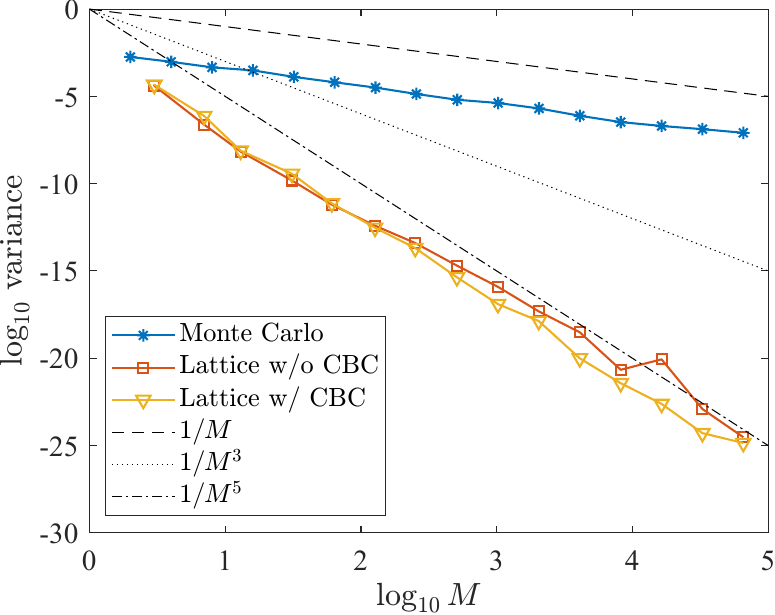}
\end{subfigure}
\hfill
\begin{subfigure}{.48\textwidth}
  \centering
  \includegraphics[width=\linewidth]{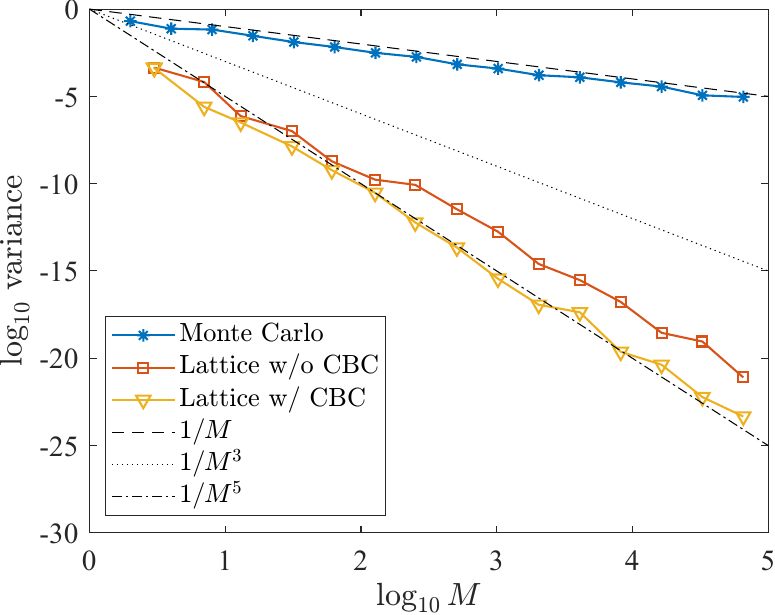}
\end{subfigure}\\[1em]
\begin{subfigure}{.48\textwidth}
  \centering
  \includegraphics[width=\linewidth]{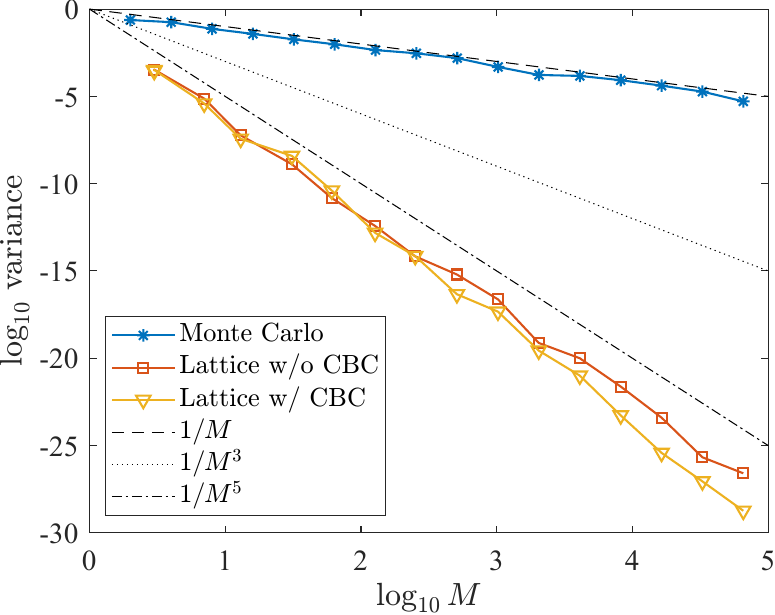}
\end{subfigure}
\hfill
\begin{subfigure}{.48\textwidth}
  \centering
  \includegraphics[width=\linewidth]{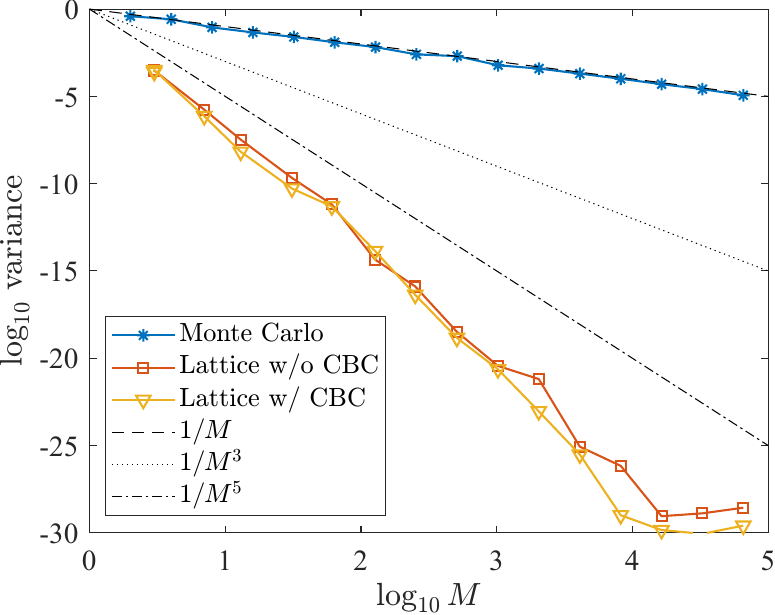}
\end{subfigure}
\caption{Convergence behavior of randomized lattice rules for test functions in higher dimension $d=20$. Other descriptions are the same as Figure~\ref{fig:test_d2}.}
\label{fig:test_d20}
\end{figure}

\bibliographystyle{amsplain}
\bibliography{ref.bib}

\end{document}